\documentclass[12pt,reqno]{amsart}
\usepackage[headings]{fullpage}
\usepackage{amssymb,amsmath,amscd,bbm,tikz-cd}
\usepackage[all,cmtip]{xy}
\usepackage{url}
\usepackage[bookmarks=true,%
    colorlinks=true,%
    linkcolor=blue,
    citecolor=blue,%
    filecolor=blue,%
    menucolor=blue,%
    urlcolor=blue,%
    breaklinks=true]{hyperref}
\usepackage{slashed}
\usepackage{listings}
\usepackage{verbatim}
\usepackage{mathtools}
\usepackage[normalem]{ulem}   
\usepackage{graphicx}
\usepackage{texdraw}
\usepackage{caption}
\graphicspath{{figures/}}

\newtheorem{theorem}{Theorem}[section]
\theoremstyle{definition}

\newtheorem{lemma}[theorem]{Lemma}
\newtheorem{definition}[theorem]{Definition}
\newtheorem{remark}[theorem]{Remark}

\newtheorem{conjecture}[theorem]{Conjecture}

\newtheorem{problem}{Problem}

\newcommand{\rvline}{\hspace*{-\arraycolsep}\vline\hspace*{-\arraycolsep}}

\def\BN{\mathbbm N}
\def\BZ{\mathbbm Z}

\def\BR{\mathbbm R}
\def\BC{\mathbbm C}

\def\calT{\mathcal T}

\def\SL{\mathrm{SL}}

\def\th{\theta}

\def\PSL{\mathrm{PSL}}
\def\PGL{\mathrm{PGL}}

\def\be{\begin{equation}}
\def\ee{\end{equation}}
\def\bea{\begin{equation*}}
\def\eea{\end{equation*}}

\def\G{\mathrm{G}}

\def\geom{\mathrm{geom}}

\def\Sym{\mathrm{Sym}}
\def\osp{\mathrm{OSp}}
\def\SM[#1,#2]{
\begin{pmatrix}
1 & #1 & \rvline & #2 \\
0 & 1 &  \rvline & 0 \\ \hline
0 & -#2 &  \rvline & 1
\end{pmatrix}
}
\def\LM[#1]{
\begin{pmatrix}
0 & -#1^{-1} & \rvline & 0 \\
#1 & 0 &  \rvline & 0 \\ \hline
0 & 0 &  \rvline & 1
\end{pmatrix}
}

\newcommand{\cT}{{\mathcal T}}
\newcommand{\Tdot}{\mathring{\cT}}

\newcommand{\Mdott}{\widetilde{M}}

\definecolor{codegreen}{rgb}{0,0.6,0}
\definecolor{codegray}{rgb}{0.5,0.5,0.5}
\definecolor{codepurple}{rgb}{0.58,0,0.82}
\definecolor{backcolour}{rgb}{0.95,0.95,0.92}

\lstdefinelanguage{PARIGP}{
  keywords={typeof, new, true, false, catch, function, return, null, catch, switch,
    var, if, in, while, do, else, case, break},
  keywordstyle=\color{blue}\bfseries,
  ndkeywords={class, export, boolean, throw, implements, import, this},
  ndkeywordstyle=\color{darkgray}\bfseries,
  identifierstyle=\color{black},
  sensitive=false,
  comment=[l]{//},
  morecomment=[s]{/*}{*/},
  commentstyle=\color{purple}\ttfamily,
  stringstyle=\color{red}\ttfamily,
  morestring=[b]',
  morestring=[b]"
}

\lstdefinestyle{code}{
    backgroundcolor=\color{backcolour},   
    commentstyle=\color{codegreen},
    keywordstyle=\color{magenta},
    numberstyle=\tiny\color{codegray},
    stringstyle=\color{codepurple},
    basicstyle=\ttfamily\footnotesize,
    breakatwhitespace=false,         
    breaklines=true,                 
    captionpos=b,                    
    keepspaces=true,                 
    numbers=left,                    
    numbersep=5pt,                  
    showspaces=false,                
    showstringspaces=false,
    showtabs=false,                  
    tabsize=2
}

\begin{document}
\title[Super-representations of 3-manifolds and torsion polynomials]{
  Super-representations of 3-manifolds and torsion polynomials}
\author{Stavros Garoufalidis}
\address{
International Center for Mathematics, Department of Mathematics \\
Southern University of Science and Technology \\
Shenzhen, China \newline
{\tt \url{http://people.mpim-bonn.mpg.de/stavros}}}
\email{stavros@mpim-bonn.mpg.de}

\author{Seokbeom Yoon}
\address{International Center for Mathematics, \\
Southern University of Science and Technology \\
Shenzhen, China \newline
{\tt \url{http://sites.google.com/view/seokbeom}}}
\email{sbyoon15@gmail.com}

\keywords{super-Teichmuller space, 3-manifolds, hyperbolic geometry, Thurston,
  representations, torsion polynomials, Thurston norm, genus, ideal triangulations,
  face matrices, Neumann-Zagier matrices, super-ptolemy variables, super-shapes,
  ortho-symplectic super-group.}


\date{12 March 2024}

\begin{abstract}
  Torsion polynomials connect the genus of a hyperbolic knot (a topological
  invariant) with the discrete faithful representation (a geometric invariant).
  Using a new combinatorial structure of an ideal triangulation
  of a 3-manifold that involves edges as well as faces, we associate a polynomial
  to a cusped hyperbolic manifold that conjecturally agrees with the $\BC^2$-torsion
  polynomial, which conjecturally detects the genus of the knot. The new
  combinatorics is motivated by super-geometry in dimension 3, and more precisely
  by super-Ptolemy assignments of ideally triangulated 3-manifolds and their 
  $\osp_{2|1}(\BC)$-representations.
\end{abstract}

\maketitle

{\footnotesize
\tableofcontents
}


\section{Introduction}
\label{sec.intro}

\subsection{Overview}

A well-known topic in geometry and topology is the study of representations of
surface groups into simple Lie groups. Recently, this topic has been extended by
replacing simple Lie groups (such as $\SL_2(\BC)$) with super-Lie groups, and
most notably by the orthosymplectic group $\osp_{2|1}(\BC)$.
These representations have been studied by at least three different points of view,
namely as character varieties, as cluster algebras, and as super-Teichm\"uller space;
see for instance~\cite{Penner, Ip-Penner, Musiker:matrix, Shemyakova} as well
as~\cite{Witten}.

All this is about surfaces. In our paper we extend this study in the context of
3-manifolds equipped with an ideal triangulation. Explicitly,
\begin{itemize}
\item[(a)]
  We introduce super-Ptolemy coordinates of 3-dimensional triangulations and prove
  that they parametrize $\osp_{2|1}(\BC)$-representations of fundamental groups of
  3-manifolds; see Section~\ref{sec.super}.
\item[(b)] 
  Using such representations, we define a 1-loop polynomial and show that it is 
  a topological invariant; see Theorem~\ref{thm.23move}.
\item[(c)]
  We show that an $\SL_2(\BC)$-representation lifts to an
  $\osp_{2|1}(\BC)$-representation if and only if the 1-loop polynomial, evaluated
  at $t=1$, vanishes; see Theorem~\ref{thm.lift}.
\item[(d)]
  We conjecture that our 1-loop polynomial coincides with the $\BC^2$-torsion
   polynomial whose degree conjecturally detects the genus of a
   knot~\cite{Dunfield:twisted}; see Conjecture~\ref{conj.dt}.
   Both polynomials have values in the trace field of  a cusped hyperbolic 3-manifold,
   and are explicitly computable. Doing so, we check our conjecture explicitly
   for the $4_1$ knot. 
\end{itemize}

\subsection{Torsion polynomials: a Thurstonian connection}

The complement $M=S^3\setminus K$ of a hyperbolic knot $K$ in 3-space has two
interesting invariants, both defined by Thurston
\begin{itemize}
\item
the genus of $K$, i.e., the least genus of all embedded spanning surfaces of $K$,
generalized to the Thurston norm on $H_2(M,\partial M,\BR)$~\cite{Thurston:norm},
\item
the discrete faithful representation $\pi_1(M) \to \PSL_2(\BC)$~\cite{Thurston}.
\end{itemize}
These two invariants, one topological and another geometric, are beautifully linked
to each other via torsion polynomials revealing, to quote
Agol--Dunfield~\cite{Agol-Dunfield}, a ``remarkable Thurstonian connection between the
topology and geometry of 3-manifolds''. Torsion polynomials are twisted versions of
the Alexander polynomial, where one twists the homology of the infinite cyclic
cover of $M$ using an $\SL_2(\BC)$-lift $\rho_\geom$ of the geometric representation
of $M$, or a
symmetric power $\Sym^{n-1}(\rho_\geom)$ of it, the corresponding polynomial being
denoted by $\tau_{M,\rho_\geom, n}(t)$. These geometric invariants are Laurent
polynomials in $t$
with coefficients in the trace field of $M$, and a key feature is that their
degrees give bounds for the genus of the knot. More precisely, one has
\be
\label{genusn}
2 \cdot \text{genus}(K) -1 \geq \frac{1}{n} \deg_t \tau_{M,\rho_\geom, n}(t)
\ee
for all $n \geq 2$. 
When $n=3$ (or $n$ being odd bigger than 1), examples show that the above bound
is not sharp, but when $n=2$, it was conjectured in~\cite{Dunfield:twisted},
for reasons that are not entirely clear, and proven in several families
that the inequality in~\eqref{genusn} becomes an equality~\cite{Agol-Dunfield}.
As Agol--Dunfield state, this is a remarkable Thurstonian connection between the
topology and geometry of 3-manifolds.


The paper concerns two seemingly unrelated problems from 3-dimensional hyperbolic
geometry and character varieties of 3-manifold groups.
Below, $M$ denotes a cusped hyperbolic 3-manifold.
\begin{problem}
\label{prob1}
Can one compute the $\BC^2$-torsion polynomial $\tau_{M,\rho,2}(t)$ of an
$\SL_2(\BC)$-representation $\rho$ of $ \pi_1(M)$ from an ideal triangulation $\calT$
of $M$?
\end{problem}

\begin{problem}
\label{prob2}
What is the geometric meaning of the set of $\SL_2(\BC)$-representations $\rho$ of
 $\pi_1(M)$ whose $\BC^2$-torsion $\tau_{M,\rho,2}(1)$ is vanishing?
\end{problem}

We will answer both problems by introducing super-Ptolemy assignments for $\calT$ 
(see Section~\ref{sec.super}) and defining a 1-loop polynomial, which is a topological
invariant (see Theorem~\ref{thm.23move}). The conjectural equality of the 1-loop
polynomial with the $\BC^2$-torsion polynomial (see Conjecture~\ref{conj.dt}) gives
a solution to the first problem. The vanishing of the 1-loop polynomial at $t=1$
gives a necessary and sufficient condition for an $\SL_2(\BC)$-representation of
$\pi_1(M)$ to lift to an $\osp_{2|1}(\BC)$-representation (see Theorem~\ref{thm.lift})
thus answering the second problem. 

Both the 1-loop and the $\BC^2$-torsion polynomials have coefficients
in the trace field of the representation and can be exactly computed, and doing so
we can confirm our conjecture for the $4_1$ knot. 
Moreover, the conjecture is proven in subsequent work (using a different set of ideas)
for all fibered cusped hyperbolic 3-manifolds~\cite{GY:1loop}. Aside from the
fibered case, Nathan Dunfield has confirmed the conjecture numerically
for all manifolds in the \texttt{OrientableCuspedCensus} and beyond by using
the methods  of \texttt{SnapPy}~\cite{snappy}. 


Our super-Ptolemy assignments lead to new combinatorics of ideal triangulations
beyond the well-known Neumann--Zagier matrices, the latter encoding which tetrahedra
wind around each edge. This newly found combinatorics involves linear equations
associated to faces and tetrahedra, dictated by the representations of $\pi_1(M)$
into $\osp_{2|1}(\BC)$; see Section~\ref{sub.super3} below for a detailed discussion.

\subsection{Super-geometry in dimension 3}
\label{sub.super3}

A byproduct of our paper is a solution to the following problem that advances 
recent work of many people on super-Riemann surfaces and their
coordinates into the third dimension. For detailed description of super-geometry
in dimension 2, super-Teichm\"{u}ller space and super-Riemann surfaces, we refer the
reader to~\cite{Witten,Penner,Ip-Penner, Aghaei-Teschner,Shemyakova,
  Musiker:matrix, Ovsienko}.

\begin{problem}
\label{prob3}
What are super-Ptolemy assignments of 3-dimensional triangulations?
\end{problem}

A solution to the above problem is detailed in Section~\ref{sec.super}. Here, we
highlight the important features of super-Ptolemy assignments in dimension 3, 
postponing their precise definitions, notations and properties for later.

Recall that a Ptolemy assignment is a map $c: \calT^1 \to \BC^\ast$ from the set
of oriented edges  of an ideal triangulation $\calT$ that satisfies
$c(-e)=-c(e)$ for all edges $e$ and the equation
\be
\label{eqn.c}
c_{01} c_{23} - c_{02} c_{13} + c_{03} c_{12} =0
\ee
for each tetrahedron, where $c_{ij}=c(e_{ij})$ and $e_{ij}$ is the $(i,j)$-edge of
a tetrahedron~\cite{GTZ15} (see also~\cite{Garoufalidis:ptolemy}). There are
canonical bijections between the sets of generically decorated
$(\SL_2(\BC),N_2)$-representations, solutions to the Ptolemy equations, and natural
$(\SL_2(\BC),N_2)$-cocycles of the truncated triangulation\footnote{ 
where $N_2=  \{ \begin{psmallmatrix} 1 & * \\ 0 & 1 \end{psmallmatrix} \}$
is the unipotent subgroup of $\SL_2(\BC)$} studied in detail in~\cite{GTZ15}.

Whereas a Ptolemy assignment $c$ assigns a nonzero complex number to every edge of
an ideal triangulation, a super-Ptolemy assignment is a pair of maps $(c,\theta)$
that assign  an invertible even element of a Grassmann algebra at each edge and an
odd element at each face of an ideal triangulation. Instead of Equation~\eqref{eqn.c},
a super-Ptolemy assignment satisfies one equation
\be
\label{eqn.sc}
c_{01} c_{23} - c_{02} c_{13} + c_{03} c_{12} + 
c_{01}  c_{03}c_{12}c_{13}c_{23}  \theta _0 \theta_2 = 0 
\ee
for each tetrahedron as well as one equation for each face 
\be
\label{eqn.stheta}
\begin{aligned}
c_{12} \theta_0 - c_{02} \theta_1 + c_{01} \theta_2 & = 0  \qquad
&
c_{13} \theta_0 - c_{03} \theta_1 + c_{01} \theta_3 & = 0 \\ 
c_{23} \theta_0 - c_{03} \theta_2 + c_{02} \theta_3 & = 0  \qquad
&
c_{23} \theta_1 - c_{13} \theta_2 + c_{12} \theta_3 & = 0 
\end{aligned}
\ee
of each tetrahedron  of $\calT$. Here $c_{ij}=c(e_{ij})$ where $e_{ij}$  is the
edge $(i,j)$ and $\theta_k=\theta(f_k)$ where $f_k$ is the face opposite to the 
vertex $k$ as in Figure~\ref{fig.tetrahedron}.

Super-Ptolemy assignments lead to a fundamental correspondence 
described by a pair of bijections
\be
\label{eqn.oto}  
\xymatrix{
\left\{\txt{Generically decorated\\$(\osp_{2|1}(\BC),N)$-reps on $M$}\right\}
\ar@{<->}[r]^-{1-1} &
P_{2|1}(\calT)\ar@{<->}[r]^-{1-1} &
\left\{\txt{Natural $(\osp_{2|1}(\BC),N)$-\\cocycles on $\Tdot$}\right\}}
\ee
which are given explicitly in Figures~\ref{fig.d2p} and~\ref{fig.p2c}.

\begin{figure}[htpb!]
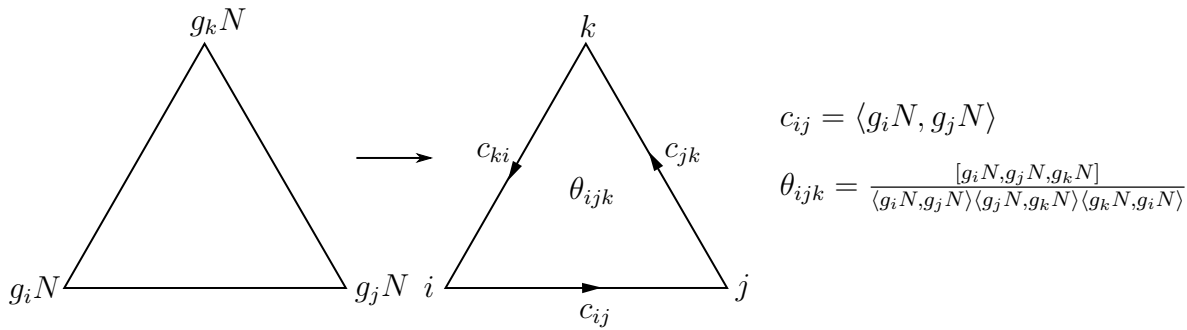
\caption{From decorated representations to Ptolemy assignments, with the bilinear
and trilinear forms as in~\eqref{bracket2} and~\eqref{bracket3}.}
\label{fig.d2p}
\end{figure}

\begin{figure}[htpb!]
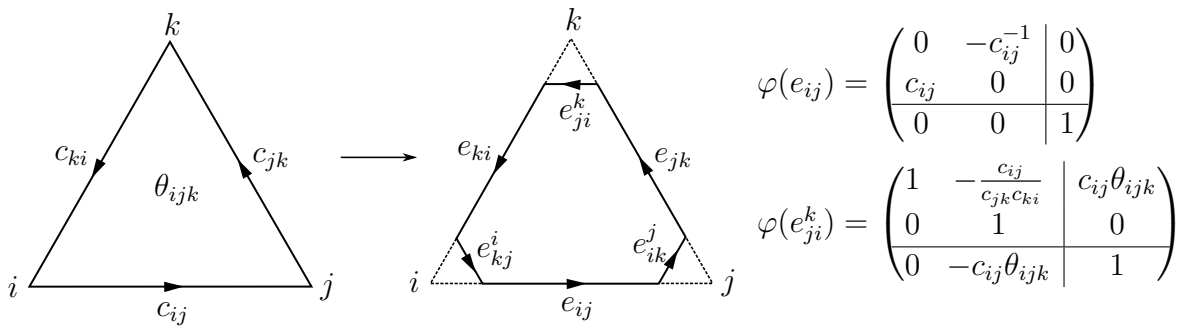
\caption{From Ptolemy assignments to natural cocycles.}
\label{fig.p2c}
\end{figure}

As written in~\eqref{eqn.stheta}, a super-Ptolemy assignment $(c,\th)$ satisfies
linear equations  in $\theta$. It turns out that these linear equations can be
written in a matrix form 
\be
F_c \, \theta =0  
\ee
where $F_c$ is a sqaure matrix whose entries are given by the Ptolemy variable $c$
with some signs. Obviously, we are interested in the case of $F_c$ being singular,
otherwise $\theta$ should be trivial. This motivates (see e.g. Section~\ref{sub.41})
the definition of a 1-loop invariant 
\be
\label{deltac2short}
\delta_{\calT,c,2} =
\frac{1}{c_1 \cdots c_N} \left(\prod_{\Delta} \frac{1}{c(e_\Delta)} \right) \det F_c
\ee
given in terms of the determinant of $F_c$. What's more, it motivates the
definition of a 1-loop polynomial
\be
\label{delta2tshort}
\delta_{\calT,c,2}(t) = \frac{1}{c_1 \cdots c_N}
\left(\prod_{\Delta} \frac{1}{c(e_\Delta)} \right) \det F_c(t)
\ee
given in terms of the determinant of a $t$-twisted version $F_c(t)$ of $F_c$
(see Section~\ref{sub.poly}).
This 1-loop polynomial $\delta_{\calT,c,2}(t)$ is unchanged under Pachner 2--3 moves
(see Theorem~\ref{thm.23move}) and its value $\delta_{\calT,c,2}(1)$ at $t=1$
determines whether the $\SL_2(\BC)$-representation $\rho$ of $\pi_1(M)$ corresponding
to the Ptolemy assignment $c$ admits an $\osp_{2|1}(\BC)$-lift or not
(see Theorem~\ref{thm.lift}). Based on the analogy with the 1-loop polynomial
$\delta_{\calT,c,3}(t)$ of~\cite{GY21}, we conjecture that the 1-loop polynomial
$\delta_{\calT,c,2}(t)$ equals to the $\BC^2$-torsion polynomial $\tau_{M,\rho,2}(t)$
(see Conjecture~\ref{conj.dt}).


\section{$\osp_{2|1}(\BC)$-representations of 3-manifolds}
\label{sec.super}

\subsection{The orthosymplectic group}
\label{sec.osp}

In this section we recall the definition of the orthosymplectic super-Lie group
$\osp_{2|1}(\BC)$. For a detailed description of super-manifolds and super-Lie
groups we refer the reader to~\cite{Berezin,Manin,Crane}.

Let $\G(\BC)$ be the Grassmann algebra over the complex numbers with unit $1$
generated by $\epsilon_i$ for $i \in \BN$: 
\be
\label{Gdef}
\G(\BC) = \BC\langle 1, \epsilon_1, \epsilon_2 ,\ldots \, | \, 1 \epsilon _i = \epsilon_i
= \epsilon_i 1,
\, \epsilon_i \epsilon_j = - \epsilon_j \epsilon_i
\textrm{ for all } i,j \in \BN  \rangle \, .
\ee
It is a $\BZ/2\BZ$-graded algebra with the unit having degree 0 and each $\epsilon_i$
having degree 1. We denote by $\G_0(\BC)$ and $\G_1(\BC)$ its even and odd part,
respectively, and  by $\G^\ast_0(\BC)$  the set of invertible elements in
$\G_0(\BC)$. There is an algebra epimorphism $\sharp  : \G(\BC) \rightarrow \BC$ 
sending all  $\epsilon_i$ to $0$, hence an element $e \in \G(\BC)$ is invertible if and 
only if $\sharp(e) \neq 0$. We write $\sharp(e)$ simply as $e^\sharp$ and call it the
body of $e$. 

An even $n|m\times n|m$-matrix $g$ is of the form
\be
g= \begin{pmatrix}
A & \rvline & B \\ \hline
C &   \rvline & D
\end{pmatrix}
\ee
where $A \in M_{n,n} (\G_0(\BC))$,  $B \in M_{n,m} (\G_1(\BC))$, $C \in M_{m,n}
(\G_1(\BC))$, and $D \in M_{m,m} (\G_0(\BC))$. The super-transpose of $g$ is
given by
\be
g^\textrm{st}= \begin{pmatrix}
A^t & \rvline & C^t \\ \hline
-B^t &   \rvline & D^t
\end{pmatrix}
\ee
and the Berezinian (or super-determinant) of $g$ is defined as
\be
\mathrm{Ber}(g) = \det (A-BD^{-1}C) \det(D)^{-1}
\ee
provided that $A$ and $D$ are invertible.

The orthosymplectic group $\osp_{2|1}(\BC)$ is the group of even $2|1 \times
2|1$-matrices $g$ satisfying
\be
\label{eqn.definition}
g^{st} \begin{pmatrix}
0 & 1 & \rvline & 0 \\
-1 & 0 & \rvline & 0 \\ \hline
0 & 0 &  \rvline & -1
\end{pmatrix} g = 
\begin{pmatrix}
0 & 1 & \rvline & 0 \\
-1 & 0 & \rvline & 0 \\ \hline
0 & 0 &  \rvline & -1
\end{pmatrix}
\ee
and $\mathrm{Ber}(g) = 1$. Writing an even $2|1\times2|1$-matrix explicitly as
\be
\label{g}
g = \begin{pmatrix}
a & b & \rvline & \alpha \\
c & d & \rvline & \beta \\ \hline
\gamma & \delta  &  \rvline & e
\end{pmatrix} \quad \textrm{for} \quad a, b, c, d,e \in \G_0 (\BC), \quad
\alpha,\beta,\gamma,\delta \in \G_1(\BC), 
\ee
the defining equations~\eqref{eqn.definition} of $\osp_{2|1}(\BC)$ are
\be
\label{eqn.defosp}
a d - b c - \gamma \delta = e^2 + 2 \alpha \beta =1, \qquad
a \beta - c \alpha - e \gamma\, =  b \beta - d \alpha  - e \delta = 0 
\ee
together with 
\be
\label{berg}
\mathrm{Ber}(g) = (a d -b c) (1 - 2 \alpha \beta e^{-2}) e^{-1} =1\,.
\ee
Note that these equations imply that $e^{\pm1} = 1 \mp \gamma \delta$; in particular,
$e^\sharp =1$. Note also that the inverse of $g$ as in~\eqref{g} is given by
\be
\label{eqn.inverse}
g^{-1} = \begin{pmatrix}
d & -b & \rvline & \delta \\
-c & a & \rvline & -\gamma\\ \hline
-\beta & \alpha  &  \rvline & e
\end{pmatrix} .
\ee

The special linear group $\SL_2(\BC)$ embeds in $\osp_{2|1}(\BC)$ in an obvious
way:
\be
\label{eqn.inj}
\SL_2 (\BC) \hookrightarrow \osp_{2|1}(\BC),  \quad
\begin{pmatrix}
a & b\\
c& d
\end{pmatrix} \mapsto
\begin{pmatrix}
a & b & \rvline & 0 \\
c & d & \rvline &  0 \\ \hline
0 & 0  & \rvline & 1
\end{pmatrix} .
\ee
Conversely, applying the epimorphism $\sharp$ entrywise, we obtain an epimorphism 
\be
\label{eqn.surj}
\osp_{2|1} (\BC) \twoheadrightarrow \SL_2(\BC),  \quad
\begin{pmatrix}
a & b & \rvline & \alpha \\
c & d & \rvline & \beta \\ \hline
\gamma & \delta  &  \rvline & e
\end{pmatrix}\mapsto
\begin{pmatrix}
a^\sharp & b^\sharp\\
c^\sharp& d^\sharp
\end{pmatrix}. 
\ee
Abusing notation, we also denote  by $\sharp$ the above epimorphism and refer to
$\sharp(g)=g^\sharp$ as the body of $g \in \osp_{2|1}(\BC)$. It follows
from~\eqref{eqn.inj} and~\eqref{eqn.surj} that for any group $G$ the epimorphism
$\sharp$ induces a surjective map
\be
\mathrm{Hom}(G,\osp_{2|1}(\BC))/_{\sim}
\overset{\sharp}{\twoheadrightarrow} \mathrm{Hom}(G,\SL_{2}(\BC))/_{\sim}
\ee
where the quotient $\sim$ is given by conjugation.

\begin{remark}
\label{rmk.oneG}
For full generality we use the Grassmann algebra with infinitely many generators as 
in~\eqref{Gdef}, but one may use one with finitely many generators.
In particular, if we use the Grassmann algebra with one odd generator
\be
\label{eqn.Godd}
\BC\langle 1, \epsilon  \, | \, 1 \epsilon  = \epsilon= \epsilon 1,
\, \epsilon^2  = 0  \rangle \, = \BC[\epsilon]/(\epsilon^2),
\ee
then its even and odd parts are $\BC$ and $\BC\epsilon$ respectively, and the
orthosymplectic group $\osp_{2|1}(\BC)$ reduces to the special affine
transformation group $\SL_2(\BC) \ltimes \BC^2$. Indeed,
if there is only one odd generator, then the map
\be
\osp_{2|1}(\BC) \rightarrow \SL_2(\BC) \ltimes \BC^2, \quad\quad
\begin{pmatrix}
a & b & \rvline & \alpha \epsilon\\
c & d & \rvline & \beta \epsilon \\ \hline
\gamma \epsilon & \delta \epsilon  &  \rvline & 1
\end{pmatrix} \mapsto \left( 
\begin{pmatrix}
a& b\\
c& d
\end{pmatrix},
\begin{pmatrix}
\alpha\\
\beta
\end{pmatrix}\right) \
\ee
is an isomorphism. 
This may seem to simplify things too much but, in fact, will be sufficient for our 
1-loop invariants--see Section \ref{sec.delta2} below.
\end{remark}

\subsection{The unipotent subgroup and pairings}
\label{sub.unip}

Since the body $g^\sharp$ of $g \in \osp_{2|1}(\BC)$ is in $\SL_2(\BC)$, the
natural action of  $\osp_{2|1}(\BC)$ on $\G_0(\BC)^2 \oplus \G_1(\BC)$
restricts to an action on $A^{2|1}$.  Here  $A^{2|1}$ is the pre-image of
$\BC^2 \setminus \{(0,0)^t\}$ under the map
\be
\G_0(\BC)^2 \oplus \G_1(\BC) \rightarrow \BC^2, \quad (a,b,\alpha)^t \mapsto 
(a^\sharp, b^\sharp)^t .
\ee
The induced action is transitive, and the stabilizer group at $(1,0,0)^t \in A^{2|1}$
is the unipotent subgroup
\be
\label{Pdef}
N=\left \{ \begin{pmatrix}
1 & b & \rvline & \alpha \\
0 & 1 &  \rvline & 0 \\ \hline
0 & -\alpha  &  \rvline & 1
\end{pmatrix} \, \Bigg| \,\, b \in \G_0(\BC), \, \alpha \in \G_1(\BC) \right \}
\ee
of $\osp_{2|1}(\BC)$. This induces a bijection
\be
\label{para}
\osp_{2|1}(\BC)/N \,\, \leftrightarrow \,\, A^{2|1}, \qquad
g N \,\, \leftrightarrow \,\, \text{left column of $g$} \,.
\ee
The space $A^{2|1} \simeq \osp_{2|1}(\BC)/N$ comes equipped with an even-valued
bilinear pairing
\be
\label{bracket2}
\langle \cdot, \cdot \rangle : A^{2|1} \times A^{2|1} \rightarrow \G_0(\BC), \qquad
\left\langle\begin{pmatrix}
a \\ b \\ \alpha
\end{pmatrix},
\begin{pmatrix}
c \\ d \\ \beta
\end{pmatrix}
\right\rangle := ad -bc -\alpha \beta 
\ee
and an odd-valued trilinear pairing
\be
\label{bracket3}
[ \cdot , \cdot  , \cdot ] : A^{2|1} \times
A^{2|1}  \times A^{2|1} \rightarrow \G_1(\BC), \quad 
\left[ 
\begin{pmatrix} a\\ b\\ \alpha \end{pmatrix},
\begin{pmatrix} b\\ c\\ \beta \end{pmatrix},
\begin{pmatrix} e\\ f\\ \gamma \end{pmatrix}
\right] :=
\det 
\begin{pmatrix}
a & b & e\\
b & c & f\\
\alpha & \beta & \gamma 
\end{pmatrix} - 2 \alpha \beta \gamma \, .
\ee
Both pairings are skew-symmetric and $\osp_{2|1}(\BC)$-invariant
\be
\langle v,w \rangle  
=\langle gv,gw \rangle, \qquad
[u,v,w]=[gu,gv,gw], \qquad g \in \osp_{2|1}(\BC)\,.
\ee

\subsection{Super-Ptolemy assignments} 
\label{sec.ptolemy}

Let $M$ be a compact 3-manifold with non-empty boundary and $\calT$ be an ideal
triangulation of its interior. We denote by $\calT^1$  and $\calT^2$ the oriented
edges and the unoriented faces of $\calT$, respectively. We first assume that $\calT$
is ordered, i.e., each tetrahedron has a vertex-ordering respecting the face-gluing, but
this condition can be relaxed; see Section~\ref{sub.concrete} below. 

\begin{definition}
\label{def.ctheta}  
A \emph{super-Ptolemy assignment} on $\calT$  is a pair of maps
\be
\label{ctheta}
c : \calT^1 \rightarrow \G^\ast_0(\BC), \qquad \theta : \calT^2 \rightarrow \G_1(\BC)
\ee
satisfying $c(-e)=-c(e)$ for all $e \in \calT^1$ and 
\be
\label{eqn.ptolemy}
c_{01} c_{23} - c_{02} c_{13} + c_{03} c_{12} + 
c_{01}  c_{03}c_{12}c_{13}c_{23}  \theta _0 \theta_2 = 0 
\ee
as well as  
\be
\label{eqn.odd}
\begin{aligned}
E_{\Delta,f_3} & : & c_{12} \theta_0 - c_{02} \theta_1 + c_{01} \theta_2 & = 0  \qquad
&
E_{\Delta,f_2} & : & 
c_{13} \theta_0 - c_{03} \theta_1 + c_{01} \theta_3 & = 0 \\ 
E_{\Delta,f_1} & : & c_{23} \theta_0 - c_{03} \theta_2 + c_{02} \theta_3 & = 0  \qquad
&
E_{\Delta,f_0} & : & 
c_{23} \theta_1 - c_{13} \theta_2 + c_{12} \theta_3 & = 0 
\end{aligned}
\ee
for each tetrahedron $\Delta$ of $\calT$. Here $c_{ij}=c(e_{ij})$ where $e_{ij}$ is
the oriented edge $[i,j]$ of $\Delta$ and $\theta_k=\theta(f_k)$ where $f_k$ is the
face of $\Delta$ opposite to the vertex $k$ as in Figure~\ref{fig.tetrahedron}.
\end{definition}

\begin{figure}[htpb!]
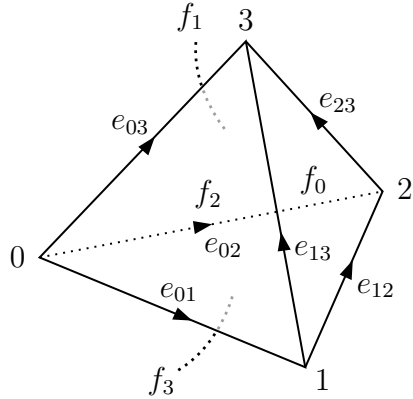
\caption{Edge and face labels for a tetrahedron.}
\label{fig.tetrahedron}
\end{figure}

\begin{lemma}
\label{lem.twofour}
If any two of \eqref{eqn.odd} together with \eqref{eqn.ptolemy} are 
satisfied, then so are the other two.
\end{lemma}

\begin{proof}
Multiplying the first equation in~\eqref{eqn.odd} by $\theta_0$ implies that 
$c_{01} \theta_0 \theta_2  = c_{02} \theta_0 \theta _1$.
Similarly, we deduce that $c_{ij}^{-1} \theta_i \theta_j$ does not depend on a
choice of $i\neq j$. It follows that Equation~\eqref{eqn.ptolemy} is equivalent to
\be
c_{01} c_{23} - c_{02} c_{13} + c_{03} c_{12} + 
c_{01} c_{02}  c_{03}c_{12}c_{13}c_{23} \, c_{ij}^{-1} \theta _i \theta_j = 0 \quad 
\textrm{for } i \neq j.
\ee
This implies that 
$(c_{01} c_{23} - c_{02} c_{13} + c_{03} c_{12}) \theta_i =0$ for all $i$.
Then one easily checks that any three out of the four equations in~\eqref{eqn.odd} 
are linearly dependent. For instance,  
\be
\label{eqn.dep}
c_{01} E_{\Delta,f_1} -c_{02} E_{\Delta,f_2} + c_{03} E_{\Delta,f_3} =0 \,.
\ee
This completes the proof.
\end{proof}

\begin{remark}
Each equation in~\eqref{eqn.odd} corresponds to a face of a tetrahedron $\Delta$.
It follows that
\be
\begin{array}{rl}
  E_{\Delta,f_0} = 0 \\
  -E_{\Delta,f_1} = 0 \\
  E_{\Delta,f_2} = 0 \\
  -E_{\Delta,f_3} = 0 \\
\end{array}
\Leftrightarrow
F_{\Delta,c} \begin{pmatrix} \theta_0 \\ \theta_1 \\ \theta_2 \\ \theta_3
  \end{pmatrix} = 0
\ee
where $F_{\Delta,c}$ is a $4 \times 4$ matrix whose rows and columns are indexed by
the faces of $\Delta$, given explicitly by
\be
\label{FD}
F_{\Delta,c}=\begin{pmatrix}
  0 & c_{23} & -c_{13} & c_{12} \\
  -c_{23} & 0 & c_{03} & -c_{02} \\
  c_{13} & -c_{03} & 0 & c_{01} \\
  -c_{12} & c_{02} & -c_{01} & 0
\end{pmatrix}  \,.
\ee
Note that $F_{\Delta,c}$ is a skew-symmetric matrix whose $(i,j)$-entry for $i \neq j$
is, up to a sign, the Ptolemy variable of the edge $f_i \cap f_j$. Note also that
$F_{\Delta,c}$ has rank $3$, whereas its body $F_{\Delta,c}^\sharp$ (the matrix
obtained by applying the epimorphism $\sharp$ to all entries of $F_{\Delta,c}$)
has rank $2$.
\end{remark}

Let $P_{2|1}(\calT)$ be the set of all  super-Ptolemy assignments on $\calT$.
Composing the epimorphism $\sharp : \G(\BC) \rightarrow \BC$ with a super-Ptolemy
assignment  $(c, \theta)$, we obtain a Ptolemy assignment $c^\sharp : \calT^1
\rightarrow \BC^\ast$  (note that $\theta$ vanishes  if we apply $\sharp$), i.e.,
$c^\sharp$ satisfies 
\be
c^\sharp_{01} c^\sharp_{23} -c^\sharp_{02} c^\sharp_{13}  + c^\sharp_{03}
c^\sharp_{12} =0
\ee
for each tetrahedron of $\calT$~\cite{GTZ15}. On the other hand,  any Ptolemy
assignment on $\calT$  forms a super-Ptolemy assignment with
the trivial map $\calT^2 \rightarrow \G_1(\BC)$, assigning 0 to all faces.
Therefore, the epimorphism $\sharp$ induces a surjective map
\be
P_{2|1}(\calT) \overset{\sharp}{\twoheadrightarrow} P_2(\calT)
\ee
where $P_2(\calT)$ is the set of all Ptolemy assignments on $\calT$.

\subsection{Natural cocycles}
\label{sub.natural}

After truncating the ideal tetrahedra of $\calT$, one obtains a cell decomposition 
$\Tdot$ of $M$; see Figure~\ref{fig.truncated}.
The 1-cells $\Tdot^1$ of $\Tdot$ consist of long and short edges, and the 2-cells
$\Tdot^2$ consist of hexagons (one for each face of $\calT$) and triangles
(one for each vertex of a tetrahedron of $\calT$). Note that the triangles of $\Tdot$
give a triangulation of the boundary $\partial M$ of $M$ and that the 2-skeleton
of $\Tdot$ defines a natural groupoid associated to $\calT$, whose generators are
the short and long edges of $\Tdot$ and relations are the triangular and hexagonal 
faces of $\Tdot$.
  
\begin{figure}[htpb!]
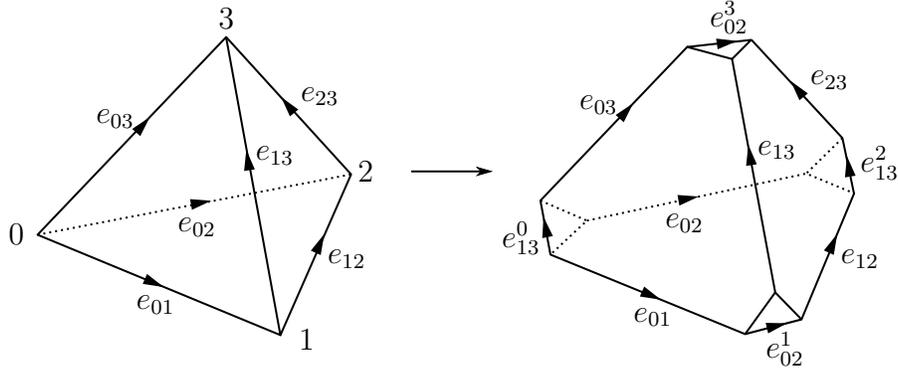
\caption{Truncating an ideal tetrahedron.}
\label{fig.truncated}
\end{figure}

\begin{definition}
\label{def.natural}
 A \emph{natural $(\osp_{2|1}(\BC),N)$-cocycle} or simply \emph{natural cocycle}
 on $\Tdot$ is a map $\varphi : \Tdot^1 \rightarrow \osp_{2|1}(\BC)$
of the form
\be
\label{eqn.natural}
\varphi(\text{short}) = 
\SM[a,\theta], \qquad
\varphi(\text{long}) =
\LM[b]
\ee
that maps the hexagons and the triangles to the identity. 
In other words, a natural cocycle is an $\osp_{2|1}(\BC)$-representation of
the groupoid of $\calT$ whose generators have the form~\eqref{eqn.natural}.
\end{definition}

Given a natural cocycle $\varphi$, let us denote
\begin{align*}
&\varphi_0(\text{short}) = \varphi(\text{short})_{1,2} \in \G_0(\BC), \qquad 
\varphi_0(\text{long}) = \varphi(\text{long})_{2,1} \in \G^\ast_0(\BC),  \\
&\varphi_1(\text{short}) = \varphi(\text{short})_{1,3} \in \G_1(\BC) \,.
\end{align*}
We now express the cocycle condition for $\varphi$ in terms of
conditions on $\varphi_0$ and $\varphi_1$.
It follows from~\eqref{eqn.inverse} that 
$\varphi(-e) = \varphi(e)^{-1}$ for $e \in \Tdot^1$  if and only if
$\varphi_i(-e) =- \varphi_i(e)$ for $i=0,1$.

\begin{lemma}
\label{lem.hexagon}
$\varphi$ satisfies the cocycle condition for a hexagon  if and only if
\be
\varphi_0 (e^k_{ji})=- \frac{\varphi_0(e_{ij})}{\varphi_0(e_{jk}) \varphi_0(e_{ki})}
\ee
for all cyclic permutations $(i,j,k)$ of $(0,1,2)$ and
\be
\label{eqn.hexagon}
\frac{\varphi_1(e^2_{10})}{\varphi_0(e_{01})} 
= \frac{\varphi_1(e^0_{21})}{\varphi_0(e_{12})} 
=\frac{\varphi_1(e^1_{02})}{\varphi_0(e_{20})}
\ee
where $e_{ij}$ and $e^k_{ij}$ denote the edges of the hexagon as in
Figure~\ref{fig.triangle}. 

\begin{figure}[htpb!]
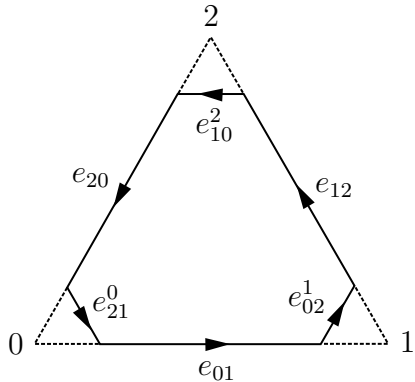
\caption{A hexagonal face.}
\label{fig.triangle}
\end{figure}
\end{lemma}

\begin{proof} 
The proof follows from a straightforward computation for the cocycle condition, i.e.,
comparing the entries of
$\varphi(e^0_{21}) \varphi(e_{01}) \varphi(e^1_{02})=
\varphi(e_{20})^{-1} \varphi(e^2_{10})^{-1} \varphi(e_{12})^{-1}$.
\end{proof}

Let $\theta \in \G_1(\BC)$ be the odd element given in Equation~\eqref{eqn.hexagon}.
It follows from Lemma~\ref{lem.hexagon} that the $\varphi_0$ and $\varphi_1$-values
on the short edges are determined by the $\varphi_0$-values on the long edges with
$\theta$. More precisely, we have
\be
\label{eqn.matrix}
\varphi(e^k_{ji}) = \SM[- \frac{\varphi_0(e_{ij})}{\varphi_0(e_{jk})
\varphi_0(e_{ki})}, \varphi_0(e_{ij}) \theta ]
\ee
for any cyclic permutations $(i,j,k)$ of $(0,1,2)$. Then identifying each long edge of $\Tdot$ with an edge of $\calT$  naturally and placing the odd element $\theta$ to the corresponding hexagonal face of $\Tdot$, or equivalently, to the
 face of $\calT$, we deduce that a natural cocycle $\varphi$
 is determined by two maps 
 \begin{equation*}
 c: \calT^1 \rightarrow \G^\ast_0(\BC), \quad 
 \theta:\calT^2 \rightarrow \G_1(\BC)\,.
 \end{equation*}
 Note that  $c$ is the restriction of $\varphi_0$ to the long edges.

\begin{lemma}
\label{lem.tetrahedron}
$\varphi$ satisfies the cocycle condition for all triangular faces of $\Tdot$ if and
only if the pair $(c,\theta)$ defined above is a super-Ptolemy assignment, i.e.,
satisfies~\eqref{eqn.ptolemy} and~\eqref{eqn.odd} for all tetrahedra of $\calT$.
\end{lemma}

\begin{proof}
Labeling the vertices of a tetrahedron with $\{0,1,2,3\}$ and using the same notation
as in Lemma~\ref{lem.hexagon}, the cocycle condition for the triangular faces
has the form $ \varphi(e^i_{jl})=\varphi(e^i_{jk})\varphi(e^i_{kl})$. For instance,
the triangular face near the vertex $0$ gives
$\varphi(e^{0}_{31})=\varphi(e^{0}_{32})\varphi(e^{0}_{21})$:
\be
\SM[- \frac{c_{13}}{c_{30} c_{01}},
c_{13} \theta_2] = \SM[- \frac{c_{23}}{c_{30} c_{02}},
c_{23} \theta_1 ]
\SM[- \frac{c_{12}}{c_{20} c_{01}},
c_{12} \theta_3]\,.
\ee
Comparing the entries of the above equation, we obtain~\eqref{eqn.ptolemy}
and the first equation in~\eqref{eqn.odd}. We obtain the other three equations
of~\eqref{eqn.odd} similarly from the other triangular faces.
\end{proof}

Lemmas~\ref{lem.hexagon} and~\ref{lem.tetrahedron} show that there is a one-to-one
correspondence
\be
\label{eqn.oto1}  
\xymatrix{P_{2|1}(\calT)\ar@{<->}[r]^-{1-1}&
\left\{\txt{Natural $(\osp_{2|1}(\BC),N)$-\\cocycles on $\Tdot$}\right\}} \, .
\ee
This verifies one bijection in the fundamental correspondence~\eqref{eqn.oto} and
its explicit formula given in Figure~\ref{fig.p2c}.

\begin{remark}
Applying the epimorphism $\sharp : \osp_{2|1}(\BC) \rightarrow \SL_2(\BC)$ to both
sides of~\eqref{eqn.oto1}, the correspondence~\eqref{eqn.oto1} reduces to the
bijection between $P_2(\calT)$ and the set of natural
$(\SL_2(\BC),N_2)$-cocycles~\cite[Sec.1.2]{Garoufalidis:ptolemy}.
 Namely, there is a commutative diagram
\be
\label{eqn.lift}
\begin{tikzcd}
P_{2|1}(\calT) \arrow[d, two heads, "\sharp"] \arrow[leftrightarrow, "1-1"]{r} &
\left\{\txt{Natural $(\osp_{2|1}(\BC),N)$-\\cocycles on $\Tdot$}\right\} 
\arrow[d, two heads,"\sharp"] \\
P_2(\calT) \arrow[leftrightarrow, "1-1"]{r} & 
\left\{\txt{Natural $(\SL_{2}(\BC),N_2)$-\\cocycles on $\Tdot$}\right\}
\end{tikzcd}
\ee
\end{remark}
\subsection{Decorations}
\label{sec.deco}

Let $\widetilde{M}$ be the universal cover of $M$ and  $\widetilde{\calT}$ be the
ideal triangulation of its interior induced from $\calT$. We denote by 
$\widetilde{\calT}^0$  the set of (ideal) vertices of $\widetilde{\calT}$. We will
use similar notations for $\Tdot$.

\begin{definition} 
\label{def.deco} 
(a) An \emph{$(\osp_{2|1}(\BC),N)$-representation} is an
$\osp_{2|1}(\BC)$-representation $\rho$ of $\pi_1(M)$  such that
$\rho (\pi_1(\partial M)) \subset N$ up to conjugation.
\newline
(b) A \emph{decoration} of an $(\osp_{2|1}(\BC),N)$-representation $\rho$ is a map
$D : \widetilde{\calT}^0 \rightarrow \osp_{2|1}(\BC)/N$ such that
\be
\label{eqn.equiv}
D(\gamma   \cdot v) = \rho(\gamma) D(v) \quad \textrm{ for } \gamma \in \pi_1(M) 
\textrm{ and }  v \in \widetilde{\calT}^0  \,.
\ee
(c) A decoration $D$ is called  \emph{generic} if $\langle D(v_0), D(v_1)
 \rangle^\sharp \neq 0$ for all vertices $v_0$ and $v_1$ joined by an edge of
 $\widetilde{\calT}$. Here we use the identification~\eqref{para} and the bilinear
 pairing~\eqref{bracket2}.
\end{definition}

In what follows, by a generically decorated representation we mean an
$(\osp_{2|1}(\BC),N)$-representation with a generic decoration. For simplicity we
identify a generically decorated representation $(\rho,D)$ with $(g \rho g^{-1}, 
gD)$  for all $g \in \osp_{2|1}(\BC)$. Note that if $D$ is a (generic) decoration of
$\rho$, then $gD$ is  a (generic) decoration of $g\rho g^{-1}$.

\begin{lemma}
\label{lem.unique}
For $N$-cosets $gN$ and $hN$ with $\langle gN,hN \rangle^\sharp \neq 0$ there
is a unique pair of coset-representatives $g' \in gN$ and $h' \in hN$ such that
$(g')^{-1} h'$ is of the form
\be
\label{eqn.counter}
(g')^{-1} h' = \LM[c] \,.
\ee
Moreover, $c = \langle gN, hN \rangle$. 
\end{lemma}

\begin{proof}
We may assume that $gN=N$ and $hN$ corresponds to $(a,c,\gamma)^t \in A^{2|1}$ 
with $c^\sharp \neq 0$.  Then a straightforward computation
\be
\begin{pmatrix}
1 & f & \rvline & \epsilon \\
0 & 1 &  \rvline & 0 \\ \hline
0 & -\epsilon &  \rvline & 1
\end{pmatrix}^{-1}
\begin{pmatrix}
a & b & \rvline & \alpha \\
c & d &  \rvline & \beta \\ \hline
\gamma &  \delta &  \rvline & e
\end{pmatrix}
=
\begin{pmatrix}
a-cf-\epsilon \gamma & b-df-\epsilon \delta & \rvline & \alpha - e \epsilon \\
c & d &  \rvline & \beta \\ \hline
c\epsilon+\gamma &  d \epsilon + \delta &  \rvline & \epsilon \beta +e
\end{pmatrix}
\ee
shows that the matrix in the right-hand side is of the form~\eqref{eqn.counter} only
if $\epsilon = -\gamma/c$, $f = a/c$, and  $d=\beta=0$.  Then it follows from
the  defining equations of $\osp_{2|1}(\BC)$ that $\delta=0$, $e=1$, $b=-1/c$, and
$\alpha =- \gamma/c$. This proves that the desired pair of coset-representatives
exists uniquely.
\end{proof}

For a generically decorated representation $(\rho,D)$ Lemma~\ref{lem.unique} implies
that there exists a unique map 
\be
\psi : \mathring{\widetilde{\calT}}^0 \rightarrow \osp_{2|1}(\BC)
\ee
such that 
\begin{itemize}
\item
$\psi(v) \in D(w)$ if $v$ is in the boundary
component of $\Mdott$ corresponding to $w \in \widetilde{\calT}^0$;
\item
$\psi(v_0)^{-1} \psi(v_1)$ is a matrix of the form~\eqref{eqn.counter} if $v_0$ and
$v_1$ are joined by an long edge. 
\end{itemize}
It follows from~\eqref{eqn.equiv} that $\psi(\gamma \cdot v) =\rho(\gamma) \psi(v)$
for $\gamma \in \pi_1(M)$, hence $\psi(\gamma \cdot v_0)^{-1} \psi(\gamma \cdot v_1)
=\psi(v_0)^{-1} \psi(v_1)$ for any vertices $v_0$ and $v_1$. Therefore if we define 
\be
\varphi : \Tdot^1 \rightarrow \osp_{2|1}(\BC), \quad 
 \varphi(e) :=\psi(v_0)^{-1} \psi(v_1)
\ee
for any lift $[v_0,v_1]$ of an edge $e \in \Tdot^1$, then $\varphi$ is well-defined 
and  by definition is a natural cocycle. This construction induces a one-to-one
correspondence
\be
\label{eqn.oto2}  
\xymatrix{
\left\{\txt{Natural $(\osp_{2|1}(\BC),N)$-\\cocycles on $\Tdot$}\right\}
\ar@{<->}[r]^-{1-1} &
\left\{\txt{Generically decorated\\$(\osp_{2|1}(\BC),N)$-reps on $M$}\right\}
.}
\ee

\begin{remark}
Applying the epimoprhism $\sharp : \osp_{2|1}(\BC) \rightarrow \SL_2(\BC)$ to both
sides of~\eqref{eqn.oto2}, the correspondence~\eqref{eqn.oto2} reduces to the
 bijection between natural $(\SL_2(\BC),N_2)$-cocycles and generically decorated
$(\SL_2(\BC),N_2)$-representations; see \cite[Sec 1.2]{Garoufalidis:ptolemy}.
\end{remark}

Combining the correspondences~\eqref{eqn.oto1} and~\eqref{eqn.oto2}, we obtain
\be
\label{eqn.oto3}  
\xymatrix{
\left\{\txt{Generically decorated\\$(\osp_{2|1}(\BC),N)$-reps on $M$}\right\}
\ar@{<->}[r]^-{1-1} &
P_{2|1}(\calT)}
\ee
given explicitly on edges $e_{ij} \in \calT^1$ and faces $f_{ijk} \in \calT^2$ by
\be
\label{d2c}
c(e_{ij}) = \langle g_i N, g_j N \rangle, \qquad
\theta(f_{ijk}) = \frac{[ g_i N, g_j N, g_k N]}{\langle g_i N, g_j N \rangle
\langle g_j N, g_kN \rangle \langle g_k N, g_i N  \rangle} \,.
\ee
Here we use the identification~\eqref{para} and
the bilinear and trilinear pairings~\eqref{bracket2} and~\eqref{bracket3}.
This verifies the first bijection in the fundamental correspondence~\eqref{eqn.oto}
and the explicit formula given in Figure~\ref{fig.d2p}. 

The next theorem is a direct consequence of the fundamental correspondence~\eqref{eqn.oto}
(see also the diagram~\eqref{eqn.lift}).

\begin{theorem}
\label{thm.ptolemy}
There is a map $P_{2|1}(\calT) \rightarrow
\mathrm{Hom}(\pi_1(M),\osp_{2|1}(\BC))/_{\sim}$ which fits into
a commutative diagram
\be
\label{eqn.diagram}
\begin{tikzcd}
P_{2|1}(\calT) \arrow[d, two heads, "\sharp"] \arrow[r] &
\mathrm{Hom}(\pi_1(M), \osp_{2|1}(\BC)) 
/ _{\sim} \arrow[d, two heads,"\sharp"] \\
P_2(\calT) \arrow[r] & \mathrm{Hom}(\pi_1(M), \SL_{2}(\BC)) / _{ \sim}
\end{tikzcd}
\ee
and whose image is the set of all conjugacy
classes of $(\osp_{2|1}(\BC),N)$-representations admitting a generic decoration.
\end{theorem}

\subsection{Action on $P_{2|1}(\calT)$}
\label{sub.action}

Let $h$ be the the number of (ideal) vertices of $\calT$. There is an action of
$\G^\ast_0(\BC)^h$ on $P_{2|1}(\calT)$ :
\be
\G^\ast_0(\BC)^h \times P_{2|1}(\calT) \rightarrow P_{2|1}(\calT),
\quad (x, (c,\theta)) \mapsto x \cdot (c,\theta)=(x \cdot c, \, x \cdot \theta) 
\ee 
where $x \cdot c$ and $x \cdot \theta$ are defined as follows.
Regarding that $x =(x_1,\ldots,x_h)$ is assigned to the vertices of $\calT$, 
\be
x \cdot c : \calT^1 \rightarrow \G^\ast_0(\BC), \quad e \mapsto x_i x_j c(e)
\ee
where $x_i$ and $x_j$ are assigned to the vertices of $e$, and 
\be
x \cdot \theta : \calT^2 \rightarrow \G_1(\BC), \quad f \mapsto (x_i x_jx_k)^{-1} \theta(f)
\ee
where $x_i, x_j,$ and $x_k$ are assigned to the vertices of $f$. One easily checks that
$x \cdot (c, \theta)$  satisfies Equations~\eqref{eqn.ptolemy} and~\eqref{eqn.odd}, 
i.e., $x \cdot (c, \theta) \in  P_{2|1}(\calT)$. This action reduces to the
$(\BC^\ast)^h$-action on $P_2({\calT})$ described
in~\cite[\S4]{Garoufalidis:ptolemy} if we forget $\theta$ and restrict $x$ 
to $(\BC^\ast)^h$.
\begin{theorem}
\label{thm.same}
The super-Ptolemy assignments $(c,\theta)$ and $x \cdot (c, \theta)$ determine 
up to conjugation the same representation.
\end{theorem}

\begin{proof} 
Let $(\rho,D)$ be a generically decorated representation 
corresponding to $(c,\theta) \in P_{2|1}(\calT)$.
Regarding $x =(x_1,\ldots,x_h)$ is assigned to the vertices of $\calT$, we define
\be
x \cdot D :  \widetilde{\calT}^0 \rightarrow \osp_{2|1}(\BC)/N, \quad v \mapsto x_i D(v)
\ee
if $v$ is a lift of the $i$-th vertex of $\calT$. Here we use the identification
$\osp_{2|1}(\BC)/N \simeq A^{2|1}$, hence the scalar multiplication
is well-defined. Then $x \cdot D$ is also a generic decoration of $\rho$, and
 Equation~\eqref{d2c} implies that   $(\rho, x \cdot D)$ corresponds to 
 $(x \cdot c, x \cdot \theta)$. 
\end{proof}

\begin{remark}
\label{rmk.reflect}
For $k \in \G^\ast_0(\BC)$  we have $(k,\ldots,k) \cdot (c, \theta) 
= (k^2c, k^{-3}\theta)$.  In particular, $(c,\theta)$ and $(c,-\theta)$ determine the 
same representation up to conjugation.
\end{remark}

\subsection{$(m,l)$-deformation}
\label{sub.deform}

One can generalize super-Ptolemy assignments and all arguments that we used in
previous sections to  $\osp_{2|1}(\BC)$-representations that may not
$(\osp_{2|1}(\BC),N)$. This can be done by considering the natural action of
$\osp_{2|1}(\BC)$  on  $A^{2|1} / \G^\ast_0(\BC)$, instead of $A^{2|1}$, where the
quotient is given by identifying $(a,b,\alpha)^t$ and $c (a,b,\alpha)^t$ for all $c \in
\G^\ast_0(\BC)$.  The stabilizer group of this action at $[(1,0,0)^t]$ is
\be
B= \left \{ \begin{pmatrix}
a & b & \rvline & \alpha \\
0 & a^{-1} &  \rvline & 0 \\ \hline
0 & -a^{-1}\alpha  &  \rvline & 1
\end{pmatrix} \, \Bigg| \,\,
a \in \G^\ast_0(\BC), \, b \in \G_0(\BC), \, \alpha \in \G_1(\BC)
\right \}.
\ee

We fix a cocycle $\sigma$ that assigns an element of $\G^\ast_0(\BC)$ to each
short edge of $\Tdot$, and consider representations
$\rho :\pi_1(M)\rightarrow \osp_{2|1}(\BC)$ satisfying up to conjugation
\be
\label{eqn.B}
\rho(\gamma) = 
\begin{pmatrix}
\sigma(\gamma) & \ast & \rvline & \ast \\
0 & \sigma(\gamma)^{-1} &  \rvline & 0 \\ \hline
0 & \ast  &  \rvline & 1
\end{pmatrix} \quad \textrm{ for all } \gamma \in \pi_1(\partial M)\, .
\ee
Here we use the same notation $\sigma$ for the cocycle and for the morphism
$\pi_1(\partial M) \rightarrow \G^\ast_0(\BC)$ induced from it; hopefully this
will cause no confusion.
As a generalization of Definitions~\ref{def.ctheta}, \ref{def.natural}, and
\ref{def.deco}, we define:

\begin{definition} 
A \emph{$\sigma$-deformed super-Ptolemy assignment} on $\calT$  is  a pair of maps
\be
c : \calT^1 \rightarrow \G^\ast_0(\BC), \qquad \theta : \calT^2 \rightarrow \G_1(\BC)
\ee
satisfying $c(-e)=-c(e)$ for all $e \in \calT^1$ and 
\be
\label{sptolemy}
c_{01} c_{23} - \frac{\sigma^2_{03}\sigma^3_{12}}{\sigma^1_{03} \sigma^0_{12}}
c_{02} c_{13} + \frac{\sigma^3_{02}\sigma^2_{13}}  { \sigma^1_{02} \sigma^0_{13}}
c_{03} c_{12} +  \frac{ \sigma^3_{02}}  { \sigma^1_{02} } 
c_{01} c_{03}c_{12}c_{13}c_{23}  \theta _0 \theta_2 = 0
\ee
as well as 
\be
\label{eqn.sood}
\begin{aligned}
E_{\Delta,f_3} \, :\,
\frac{\sigma^1_{23}}{\sigma^0_{23}} c_{12} \theta_0 - c_{02} \theta_1 +
\frac{\sigma^0_{12}}{\sigma^3_{12}} c_{01} \theta_2 & = 0 & \quad
E_{\Delta,f_2} \, :\,
c_{13} \theta_0 - \frac{\sigma^3_{01}}{\sigma^2_{01}}  c_{03} \theta_1 +
\frac{\sigma^0_{12}}{\sigma^3_{12}} c_{01} \theta_3  & = 0
\\
E_{\Delta,f_1} \, :\,
\frac{\sigma^1_{03}}{\sigma^2_{03}} c_{23} \theta_0 -
\frac{\sigma^3_{01}}{\sigma^2_{01}}  c_{03} \theta_2 + c_{02} \theta_3 & = 0
& \quad 
E_{\Delta,f_0} \, :\,
\frac{\sigma^1_{03}}{\sigma^2_{03}} c_{23} \theta_1 - c_{13} \theta_2 +
\frac{\sigma^1_{23}}{\sigma^0_{23}} c_{12} \theta_3 & = 0
\end{aligned}
\ee
for each tetrahedron $\Delta$ of $\calT$. Here $\sigma^i_{jk} \in \G^\ast_0(\BC)$ is
the element assigned by $\sigma$ at the short edge that is near to the vertex $i$
and parallel to the edge $[j,k]$; see Figure~\ref{fig.triangle}.
\end{definition}

\begin{definition}
A \emph{$\sigma$-deformed natural cocycle} on $\Tdot$ is map $\varphi : \Tdot^1
\rightarrow \osp_{2|1}(\BC)$ of the form
\be
\label{eqn.natural2}
\varphi(\text{short}) = 
\begin{pmatrix}
\sigma(\text{short}) & a& \rvline & \theta \\
0 & \sigma(\text{short})^{-1} & \rvline &  0 \\ \hline
0& -\sigma(\text{short})^{-1} \theta  & \rvline & 1
\end{pmatrix} , \qquad
\varphi(\text{long}) =
\LM[b]
\ee
that maps the hexagons and the triangles to the identity. 
In other words, a natural cocycle is an $\osp_{2|1}(\BC)$-representation of
the groupoid of $\calT$ whose generators have the form~\eqref{eqn.natural2}.
\end{definition}

\begin{definition} 
\label{def.deco2} 
(a) An \emph{$(\osp_{2|1}(\BC),B)$-representation} is an
$\osp_{2|1}(\BC)$-representation $\rho$ of $\pi_1(M)$  such that
$\rho (\pi_1(\partial M)) \subset B$ up to conjugation.
\newline
(b) A \emph{decoration} of an $(\osp_{2|1}(\BC),B)$-representation $\rho$ is a map
$D : \widetilde{\calT}^0 \rightarrow \osp_{2|1}(\BC)/B$ satisfying~\eqref{eqn.equiv}. 
\newline
(c) A decoration $D$ is called  \emph{generic} if $\langle D(v_0), D(v_1)
\rangle^\sharp \neq 0$ for all vertices $v_0$ and $v_1$ joined by an edge of
$\widetilde{\calT}$. Note that the condition makes sense, even though $\langle
 D(v_0), D(v_1) \rangle$ can be only defined up to $\G_0^\ast(\BC)$.
\end{definition}

Repeating the same arguments in Sections~\ref{sec.ptolemy}--\ref{sec.deco}
(see also \cite[\S2]{Yoon:volume}), we obtain 
\be
\label{eqn.otod}
\xymatrix{
  \left\{\txt{Generically decorated\\$(\osp_{2|1}(\BC),B)$-reps on $M$\\
      satisfying~\eqref{eqn.B}}\right\}
\ar@{<->}[r]^-{1-1} &
P^\sigma_{2|1}(\calT)\ar@{<->}[r]^-{1-1} &
\left\{\txt{$\sigma$-deformed natural \\cocycles on $\Tdot$}\right\}} 
\ee
where $P^\sigma_{2|1}(\calT)$ is the set of all $\sigma$-deformed super-Ptolemy
assignments on $\calT$.  We note that:

\noindent
1. The same argument used in Lemma~\ref{lem.twofour} shows that any three
of~\eqref{eqn.sood} are linearly dependent. 
\smallskip

\noindent
2. Composing the epimorphism $\sharp : \G(\BC) \rightarrow \BC$ with
$(c, \theta)\in P^\sigma_{2|1}(\calT)$, we obtain a $\sigma^\sharp$-deformed
Ptolemy assignment $c^\sharp : \calT^1 \rightarrow \BC^\ast$, i.e., $c^\sharp$
satisfies 
\be
c_{01}^\sharp c_{23}^\sharp - \frac{(\sigma^2_{03})^\sharp(\sigma^3_{12})^\sharp}
{(\sigma^1_{03})^\sharp(\sigma^0_{12})^\sharp}c_{02}^\sharp c_{13}^\sharp +
\frac{(\sigma^3_{02})^\sharp(\sigma^2_{13})^\sharp}
{(\sigma^1_{02})^\sharp(\sigma^0_{13})^\sharp} c_{03}^\sharp c_{12}c_{01}^\sharp=0
\ee
for each tetrahedron of $\calT$~\cite{Yoon:volume}. This defines a  
map
\be
P^\sigma_{2|1}(\calT) \overset{\sharp}{\rightarrow} P^{\sigma^\sharp}_2(\calT)
\ee
where $P_2^{\sigma^\sharp}(\calT)$ is the set of all
$\sigma^\sharp$-deformed Ptolemy assignments on $\calT$. This map is surjective
if $\sigma^\sharp = \sigma$, i.e. $\sigma$ takes values in $\BC^\ast$.
\smallskip

\noindent
3. The $\sigma$-deformed natural cocycle $\varphi$ corresponding to
$(c,\theta) \in P^\sigma_{2|1}(\calT)$ is explicitly given by
\be
\label{eqn.matrix2}
\varphi(e^k_{ji}) = 
\begin{pmatrix}
\sigma^k_{ji} & -\frac{\sigma^{j}_{ik}}{\sigma^{i}_{kj}}
\frac{c_{ij}}{c_{jk} c_{ki}} & \rvline &
c_{ij}\theta/\sigma^i_{kj} \\
0 & 1/\sigma^k_{ji} &  \rvline & 0 \\ \hline
0 & -  c_{ij}\theta / (\sigma^i_{kj}\sigma^k_{ji})  &  \rvline & 1
\end{pmatrix}
,\quad
\varphi(e_{ij})= \LM[c_{ij}]
\ee
for  Figure~\ref{fig.triangle} where $(i,j,k)$ is  a cyclic permutation of $(0,1,2)$.
\smallskip

\noindent
4. There is a $\G^\ast_0(\BC)^h$-action on $P^\sigma_{2|1}(\calT)$ defined in the
same way as that on $P_{2|1}(\calT)$; see Section~\ref{sub.action}. Also,
Theorem~\ref{thm.same} holds for  $(c,\theta) \in P^\sigma_{2|1}(\calT)$:
the $\sigma$-deformed super-Ptolemy assignments $(c,\theta)$ and
$x \cdot (c,\theta)$ determine the same representation up to conjugation.

The next theorem is a direct consequence of the correspondence~\eqref{eqn.otod}.

\begin{theorem}
There is a map $P^\sigma_{2|1}(\calT) \rightarrow
\mathrm{Hom}(\pi_1(M),\osp_{2|1}(\BC))/_{\sim}$ which fits into
\be
\label{eqn.diagram2}
\begin{tikzcd}
P^\sigma_{2|1}(\calT) \arrow[d, "\sharp"] \arrow[r] &
\mathrm{Hom}(\pi_1(M), \osp_{2|1}(\BC)) / _{\sim}
\arrow[d, two heads,"\sharp"] \\
P^{\sigma^\sharp}_{2}(\calT) \arrow[r] & \mathrm{Hom}(\pi_1(M), \SL_{2}(\BC))
/ _{ \sim}
\end{tikzcd}
\ee
and whose image is the set of all conjugacy classes of
$(\osp_{2|1}(\BC),B)$-representations admitting a generic decoration
and satisfying~\eqref{eqn.B}. 
\end{theorem}

\subsection{Concrete triangulations}
\label{sub.concrete}

In the previous Section~\ref{sub.deform} we defined the super-Ptolemy assignements
for ordered triangulations. In this section we discuss how to define these
assignments for concrete triangulations, that is, for triangulations where each
tetrahedron comes with a bijection of its vertices with those of the standard
3-simplex. Concrete triangulations were used in~\cite[Sec.2]{Garoufalidis:gluing}
to define the gluing equations of $\mathrm{PGL}_n(\BC)$ representations of $\pi_1(M)$, 
as well as in~\cite[Sec.2.1]{Garoufalidis:ptolemy} to define the Ptolemy variety,
with or without an obstruction class. Note that all the triangulations
in \texttt{SnapPy} and \texttt{Regina} are concrete~\cite{snappy, regina}. 

Recall that the Ptolemy variables are assigned to oriented edges and satisfy the
relation that reversing the orientation of an edge reverses the value of the
Ptolemy variable. In a concrete triangulation, several edges of tetrahedra before
face-pairings are identified with each other, and this identification introduces
signs when the edges are identified in an orientation-reversed way. For a concrete
example, see Examples 3.1.1 and 3.1.2 of~\cite{Garoufalidis:ptolemy}, for the
case of the trivial obstruction class, and Example 3.2.1
of~\cite{Garoufalidis:ptolemy} for a non-trivial obstruction class. 

We now discuss the presence of these signs in a systematic way. A face-pairing of
an ordered triangulation glues two faces preserving  vertex-orderings, hence it
preserves the orientations of the faces induced from  vertex-orderings.  Therefore,
it is sufficient for ordered triangulations to consider faces with vertex-order in
counterclockwise as in Figure~\ref{fig.triangle}. However, as  concrete
triangulations do not have such a property, both faces with vertex-order in
clockwise and ones with counterclockwise appear. Therefore, we consider both
sides of a face with each side having one odd element, i.e., we assign one odd
element to each oriented face. This seems to double the number of odd elements,
but in fact,  two odd elements $\theta$ and $\theta'$ assigned to the front and
back sides of a face (as in Figure~\ref{fig.revtri}) determine each other.

\begin{figure}[htpb!]
	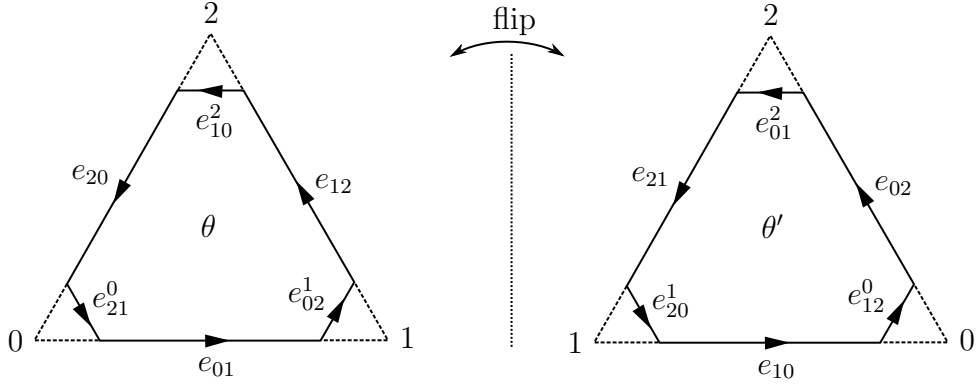
	\caption{Front/back sides of a face}
	\label{fig.revtri}
\end{figure}

Recall \eqref{eqn.matrix2} that for any cyclic permutation $(i,j,k)$ of $(0,1,2)$ 
we have
\be
\varphi(e^k_{ji}) = 
\begin{pmatrix}
	\sigma^k_{ji} & -\frac{\sigma^{j}_{ik}}{\sigma^{i}_{kj}}
	\frac{c_{ij}}{c_{jk} c_{ki}} & \rvline &
	c_{ij}\theta/\sigma^i_{kj} \\
	0 & 1/\sigma^k_{ji} &  \rvline & 0 \\ \hline
	0 & -  c_{ij}\theta / (\sigma^i_{kj}\sigma^k_{ji})  &  \rvline & 1
\end{pmatrix} \,.
\ee
Applying the same formula to the back side of the face, we obtain
\be
\varphi(e^k_{ij}) = 
\begin{pmatrix}
	\sigma^k_{ij} & -\frac{\sigma^{i}_{jk}}{\sigma^{j}_{ki}}
	\frac{c_{ji}}{c_{ik} c_{kj}} & \rvline &
	c_{ji} \theta'/\sigma^j_{ki} \\
	0 & 1/\sigma^k_{ij} &  \rvline & 0 \\ \hline
	0 & -  c_{ji} \theta' / (\sigma^j_{ki}\sigma^k_{ij})  &  \rvline & 1
\end{pmatrix} \,.
\ee
Note that $c_{ij} = - c_{ji}$ and $\sigma^{i}_{jk} = 1/ \sigma^{i}_{kj}$.
Then a straightforward computation shows that $\varphi(e^k_{ji}) \varphi(e^k_{ij}) = I$
if and only if
\be
\theta' = \sigma^{k}_{ij} \sigma^{i}_{jk} \sigma^{j}_{ki} \theta\, .
\ee
This shows that super-Ptolemy assignments on a concrete triangulation are described by the same equations~\eqref{sptolemy} and~\eqref{eqn.sood} but some of $\theta_i$ may
be replaced by $\theta'_i$. Note that $\theta_0,\ldots, \theta_3$ in~\eqref{eqn.sood} are
assigned to the sides of $f_0,\ldots,f_3$ that face front. Note also that if a face-pairing
preserves the orientation of the faces induced from the vertex-orderings, then only one
of $\theta_i$ or $\theta'_i$ appears in the face equations, and otherwise, both $\theta_i$
and $\theta'_i$ appear.

\subsection{Example: the $4_1$ knot}
\label{sub.41}

Let $\calT$ be the standard ideal triangulation of the knot complement of $4_1$
obtained by the face-pairings of two ordered ideal tetrahedra $\Delta_1$ and
$\Delta_2$ with edges $e_1$ and $e_2$ and with faces $f_1,\ldots,f_4$. See 
Figure~\ref{fig.example}. We choose a cocycle $\sigma$ on the short edges
for $m, \ell \in \G^\ast_0(\BC)$ as follows (see \cite[Ex.2.8]{Yoon:volume}):
\begin{align}
& \sigma(s_2) = \sigma(s_5) = \sigma(s_8)= \sigma(s_{11})=m, \quad
\sigma(s_6) = \sigma(s_9) = \sigma(s_{12})= m^{-1}\\
& \sigma(s_4) = \sigma(s_7) = \sigma(s_{10}) = 1, \quad
\sigma(s_1) = \ell^{-1} m^{-2}, \quad  \sigma(s_3) = \ell m
\end{align}
where $s_1,\ldots,s_{12}$ are the short edges of $\Tdot$ as in
Figure~\ref{fig.example}. Note that the morphism  induced by $\sigma$ sends 
the meridian and canonical longitude of the knot  to $m$ and $\ell$, respectively.
\begin{figure}[htpb!]
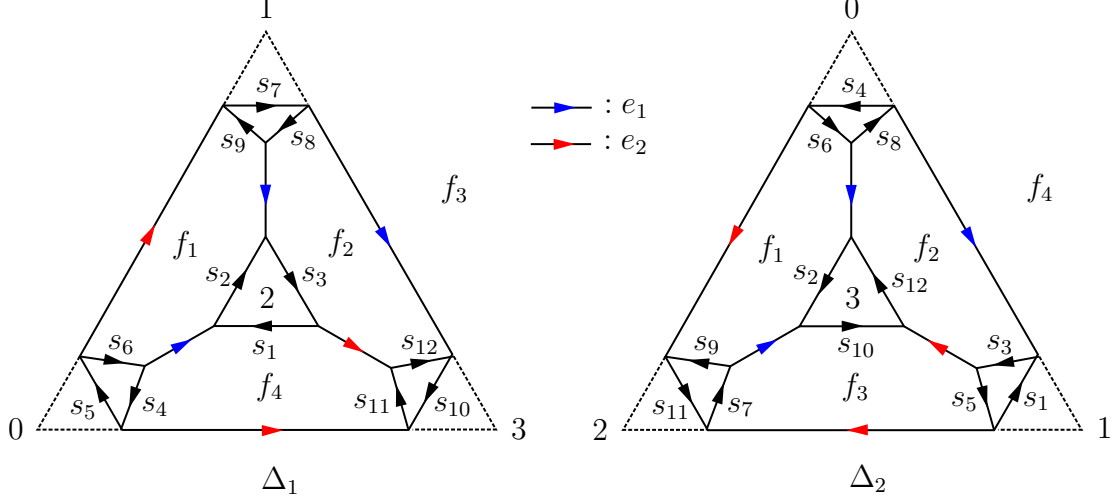
\caption{The knot complement of $4_1$.}
\label{fig.example}
\end{figure}

A $\sigma$-deformed super-Ptolemy assignments is a pair of maps
$c : \{e_1,e_2\} \rightarrow \G^\ast_0(\BC)$ and
$\theta : \{f_1,\ldots,f_4\} \rightarrow \G_1(\BC)$ satisfying
\be
\label{eqn.ex1}
\begin{aligned}
c_2^2 - \ell m^4 c_1^2 + lm^2 c_1 c_2+ m^2 c_1^3 c_2^2 \theta_2 \theta_3 &= 0\\
c_1^2 - \ell^{-1} c_2^2 +\ell^{-1} c_1 c_2 +\ell^{-1}m^{-1} c_1^3 c_2^2
\theta_3 \theta_2 &= 0
\end{aligned}
\ee
and
\be
\label{eqn.ex2}
\begin{aligned}
E_{\Delta_1,f_4}&:&
\ell^{-1}m^{-2} c_2 \theta_2 - m^{-1} c_2 \theta_3 + c_1 \theta_1 = 0  \\
E_ {\Delta_1,f_3}&: &c_1 \theta_2 -m^{-1} c_2  \theta_4  +m^{-2} c_2  \theta_1 = 0  \\ 
E_{\Delta_2,f_4}&:& c_2 \theta_3 - c_2 \theta_1 +c_1 \theta_2 = 0  \\
E_{\Delta_2,f_3}&:& \ell c_1 \theta_1 - c_2 \theta_2 + c_2 \theta_4 = 0 
\end{aligned}
\ee
where $c_i : = c(e_i)$ and $\theta_i := \theta(f_i)$. 
Writing the equations in~\eqref{eqn.ex2} in a matrix form, we have
\be
\label{eqn.exF}
\begin{pmatrix}
m^{-2} c_2 & c_1 & 0 & -m^{-1} c_2\\
c_1 & \ell^{-1} m^{-2} c_2 & -m^{-1} c_2 & 0 \\
-c_2 & c_1 & c_2 & 0 \\
\ell c_1 & -c_2 & 0 & c_2 
\end{pmatrix}
\begin{pmatrix}
\theta_1\\
\theta_2\\
\theta_3\\
\theta_4
\end{pmatrix} = \begin{pmatrix}
0\\0\\0\\0
\end{pmatrix} \,.
\ee
We are interested in the case of the above $4 \times 4$-matrix $F$ being singular,
as all $\theta_i$ should be zero, otherwise. 
One computes that if $\det F=0$, then the kernel $F$ is a free $\G_1(\BC)$-module
of rank 1:
\begin{equation*}
(\theta_1,\theta_2,\theta_3,\theta_4) =
\eta \left(\frac{c_1 + \frac{1}{\ell m} c_2}{m c_1 - c_2}, \, 1 , \,
-\frac{m c_1^2+\frac{1}{\ell m} c_2^2 }{c_2 (m c_1 -c_2) },  \,
\frac{\ell c_1^2  + (m+\frac{1}{m}) c_1 c_2 - c_2^2 }{ c_2(m c_1 -c_2 )} \right),
\quad \eta \in \G_1(\BC).
\end{equation*}
It follows that  either $\det F=0$ or not, we have $\theta_i \theta_j=0$ for any $i,j$
and thus Equation~\eqref{eqn.ex1} is simplified to 
\be
\begin{aligned}
c_2^2 - \ell m^4 c_1^2 + lm^2 c_1 c_2 &= 0\\
c_1^2 - \ell^{-1} c_2^2 +\ell^{-1} c_1 c_2 & = 0 
\end{aligned}
\ee
with
\be
\label{eqn.detF}
\det F=2 c_1^1 c_2 ^3 \, m^{-2} (m +m^{-1}-1) \,.
\ee
This shows that a $\sigma$-deformed Ptolemy assignment $(c,\theta)$ with $\theta
 \neq 0$ exists if and only if $m+m^{-1}-1=0$. For instance, we restrict $m$ and
$\ell$ to complex numbers, then
\be
m=\frac{1 \pm \sqrt{-3}}{2}, \quad \ell =-1, \quad (c_1,c_2)
= k \left( \frac{1 \mp \sqrt{-3}}{2},1\right) \quad
\textrm{for } k \in \G^\ast_0(\BC) \, .
\ee
Equation~\eqref{eqn.detF} agrees with the fact that the $\BC^2$-torsion of the
knot $4_1$ is 2$(m+m^{-1}-1)$, as shown by Kitano~\cite{Kitano}.


\section{1-loop and $\BC^2$-torsion  polynomials}
\label{sec.delta2}

\subsection{The face-matrix of an ideal triangulation}
\label{sub.face}

In this section we define the 1-loop invariant, the 1-loop polynomial, and their
$(m,l)$-deformed version from an ideal triangulation, and conjecture that it agrees
with the corresponding version of the torsion polynomial. We will give the
definition of the 1-loop invariant in its three flavors in separate sections below.

As mentioned in Remark~\ref{rmk.oneG}, in this section, we use the Grassmann 
algebra with one odd generator; its even and odd part are isomorphic to $\BC$, 
and the  product of any two odd elements is zero. Then  Equation~\eqref{eqn.ptolemy}
reduces to the ordinary Ptolemy equation
\be
\label{eqn.pt}
c_{01} c_{23} - c_{02} c_{13} + c_{03} c_{12} = 0\,.
\ee
Therefore, a super-Ptolemy assignment  $(c,\theta)$ on an ideal triangulation $\calT$ 
is given by a pair of a Ptolemy assignment $c : \calT^1\rightarrow \BC^\ast$ with a 
map $\theta : \calT^2 \rightarrow \BC$ satisfying Equation~\eqref{eqn.odd} for each
tetrahedron of $\calT$. If $\calT$ has $N$ tetrahedra, then it has $N$ edges and 
$2N$ faces. Hence a super-Ptolemy assignment on $\calT$ is represented by a tuple
$(c_1,\ldots,c_N)$ of non-zero complex numbers satisfying the Ptolemy
equation~\eqref{eqn.pt} for each tetrahedron and  a tuple
$(\theta_1,\ldots,\theta_{2N})$  of complex numbers satisfying four linear
equations~\eqref{eqn.odd} for each tetrahedron. We call these linear equations
face-equations and write them in matrix form as 
\be
\label{F0123}
\begin{pmatrix}
F_c^0 \\ F_c^1 \\ F_c^2 \\ F_c^3 
\end{pmatrix}
\theta = 0, \qquad \theta=(\theta_1,\ldots,\theta_{2N})^t
\ee
where $F_c^k$ for $k=0,1,2,3$ are $N \times 2N$ matrices whose rows and columns 
are indexed by the tetrahedra and the faces of $\calT$, respectively. 
However, it was shown in Lemma~\ref{lem.twofour} that at each tetrahedron any
 three of the  linear equations~\eqref{eqn.odd} are dependent. As an equation
in~\eqref{eqn.odd}  corresponds to a face of a tetrahedron $\Delta$, a choice of two
equations 
in~\eqref{eqn.odd} is characterized by a common edge $e_\Delta$ of the two
corresponding faces. Choosing an edge $e_\Delta$ for every tetrahedron 
$\Delta$, we create a $2N \times 2N$ matrix $F_c$, called a face-matrix, 
so that the face-equations for $\theta$ take the form
\be
\label{eqn.F}
F_c \, \theta =0  \, .
\ee
Note that $F_c$ is a trimmed version of the $4N \times 2N$ matrix of
Equation~\eqref{F0123} and that entries of $F_c$ are linear forms on $c$
(in fact, the nonzero entries at up to sign, equal the value of $c$ on an edge,
see Equation~\eqref{FD}).


\subsection{1-loop invariant}
\label{sub.C2torsion}

We now have all the ingredients to define the 1-loop invariant.

\begin{definition}
For a Ptolemy assignment $c$ on $\calT$ we define the 1-loop invariant by
\be
\label{delta2}
\delta_{\calT,c,2}:= 
\frac{1}{c_1 \cdots c_N} \left(\prod_{\Delta} \frac{1}{c(e_\Delta)} \right)
\det F_c \, .
\ee
\end{definition}
Note that  $\delta_{\calT,c,2}$ is a degree 0-function of $c$, i.e., invariant under
scaling each $c_i$ to $k  c_i$ for  $k \in \BC^\ast$. 
It turns out that  $\delta_{\calT,c,2}$ does not depend on the choice of
edge $e_\Delta$ and is invariant under 2--3 Pachner moves. This follows from
the specialization of Lemma~\ref{lem.welld} and Theorem~\ref{thm.23move} at $t=1$ 
below.

We conjecture that the 1-loop invariant $\delta_{\calT,c,2}$ agrees with the
$\BC^2$-torsion $\tau_{M,\rho,2}$ of $\rho$ (also called the Reidemeister torsion
associated to $\rho$) up to sign, where $\rho:\pi_1(M)\rightarrow \SL_2(\BC)$ is
a representation associated to the Ptolemy assignment $c$.
This is the specialization at $t=1$ of Conjecture~\ref{conj.dt} below. 

\subsection{1-loop polynomial}
\label{sub.poly}

In this section we upgrade our 1-loop invariant to the 1-loop polynomial. 

We assume that $M$ has an infinite cyclic cover $\widetilde{M}$ and denote
by $\widetilde{\calT}$ the ideal triangulation of $\widetilde{M}$ induced from $\calT$.
 We identify the deck transformation group of $\widetilde{M}$ with $\{ t^k \, | 
\, k \in \BZ\}$ for a formal variable $t$ and choose a lift of every cell of $\calT$ to
$\widetilde{\calT}$. Then every cell of $\widetilde{\calT}$ is uniquely represented by 
a cell of $\calT$ with  a monomial in $t$; for instance, a face of $\widetilde{\calT}$
is represented by $t^k \cdot f$ for a face $f$ of $\calT$ and $k \in \BZ$.
Recall that a face-equation is of the form
\be
\label{eqn.face}
c_{\alpha} \, \theta (f_0) + c_{\beta} \, \theta (f_1) + c_{\gamma} \, \theta (f_2) = 0 
\ee
where $f_0, f_1$, and $f_2$ are three faces of a tetrahedron $\Delta$. Since the lift 
of $\Delta$ has three faces $t^{k_i} \cdot \widetilde{f}_i$  for some $k_i  \in \BZ$
($i=0,1,2$), we can formally modify Equation~\eqref{eqn.face} as
\be
\label{eqn.twistface}
c_{\alpha} t^{k_0} \, \theta (f_0) + c_{\beta} t^{k_1}
\, \theta(f_1) + c_{\gamma} \, t^{k_2} \, \theta (f_2) = 0 \, . 
\ee
The effect of this insertion of monomials in $t$  leads to a twisted face-matrix
$F_c(t)$.

\begin{definition}
\label{def.delta2t}  
For a Ptolemy assignment $c$ on $\calT$  we define the 1-loop polynomial by
\be
\label{delta2t}
\delta_{\calT,c,2}(t):=
\frac{1}{c_1 \cdots c_N} \left(\prod_{\Delta} \frac{1}{c(e_\Delta)}
\right) \det F_c(t) \, .
\ee
\end{definition}
It is clear that  $\delta_{\calT,c,2}=\delta_{\calT,c,2}(1)$.  
In addition, $\delta_{\calT,c,2}(t)$ determines the 1-loop invariant
$\delta_{\calT^{(n)},c,2}$ of all cyclic $n$-covers $M^{(n)}$ of $M$. This follows
by arguments similar to the ones presented in~\cite{GY21} (for the 1-loop
polynomial $\delta_{\calT,c,3}(t)$) and will not be repeated here. 

\begin{lemma}
\label{lem.welld} 
The 1-loop polynomial $\delta_{\calT,c,2}(t)$ does not depend on the choice of edge
 $e_\Delta$.
\end{lemma}

\begin{proof} 
It suffices to compare two different edge-choices for one tetrahedron $\Delta$.
Comparing $e_\Delta = [0,1]$ and $e_\Delta=[0,2]$, we have
(cf. Equation~\eqref{eqn.dep})
\be
\label{eqn.indep}
\begin{pmatrix}
c_{02} & -c_{03} \\
0 & 1
\end{pmatrix}
\begin{pmatrix}
c_{13} & - c_{03} & 0 & c_{01} \\
c_{12} & - c_{02} & c_{01} & 0 
\end{pmatrix}
=\begin{pmatrix}
c_{01} & 0 \\
0 & 1
\end{pmatrix}
\begin{pmatrix}
c_{23} & 0 & - c_{03} & c_{02} \\
c_{12} & - c_{02} & c_{01} & 0 
\end{pmatrix} \, .
\ee
The insertion of monomials in $t$ affects on both sides of~\eqref{eqn.indep}, but it
is given by multiplying the same diagonal matrix (with diagonal in monomials in $t$)
on the right. This proves that $\delta_{\calT,c,2}(t)$ is unchanged. 
\end{proof}

\begin{theorem}
\label{thm.23move}
The 1-loop polynomial  $\delta_{\calT,c,2}(t)$ is invariant under 2--3 Pachner moves. 
\end{theorem}

\begin{proof}
Suppose that $\calT$ has two tetrahedra $[0,2,3,4]$ and $[1,2,3,4]$ with a common
face $[2,3,4]$ as in Figure~\ref{fig.pachner}. Let $\calT'$  denote the ideal
triangulation obtained by replacing these two tetrahedra by $[0,1,2,3]$,
$[0,1,3,4]$, and $[0,1,2,4]$.
  
\begin{figure}[htpb!]
\begingroup%
  \makeatletter%
  \providecommand\color[2][]{%
    \errmessage{(Inkscape) Color is used for the text in Inkscape, but the package 'color.sty' is not loaded}%
    \renewcommand\color[2][]{}%
  }%
  \providecommand\transparent[1]{%
    \errmessage{(Inkscape) Transparency is used (non-zero) for the text in Inkscape, but the package 'transparent.sty' is not loaded}%
    \renewcommand\transparent[1]{}%
  }%
  \providecommand\rotatebox[2]{#2}%
  \newcommand*\fsize{\dimexpr\f@size pt\relax}%
  \newcommand*\lineheight[1]{\fontsize{\fsize}{#1\fsize}\selectfont}%
  \ifx\svgwidth\undefined%
    \setlength{\unitlength}{136.81011651bp}%
    \ifx\svgscale\undefined%
      \relax%
    \else%
      \setlength{\unitlength}{\unitlength * \real{\svgscale}}%
    \fi%
  \else%
    \setlength{\unitlength}{\svgwidth}%
  \fi%
  \global\let\svgwidth\undefined%
  \global\let\svgscale\undefined%
  \makeatother%
  \begin{picture}(1,1.2683596)%
    \lineheight{1}%
    \setlength\tabcolsep{0pt}%
    \put(0.49321375,0.00116757){\makebox(0,0)[lt]{\lineheight{1.25}\smash{\begin{tabular}[t]{l}$0$\end{tabular}}}}%
    \put(0,0){\includegraphics[width=\unitlength,page=1]{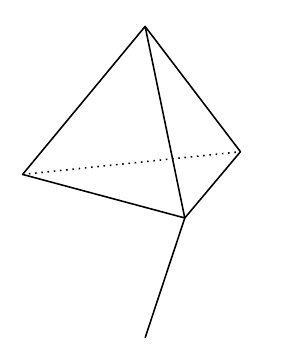}}%
    \put(0.49207883,1.21539757){\makebox(0,0)[lt]{\lineheight{1.25}\smash{\begin{tabular}[t]{l}$1$\end{tabular}}}}%
    \put(0.00575528,0.63509476){\makebox(0,0)[lt]{\lineheight{1.25}\smash{\begin{tabular}[t]{l}$2$\end{tabular}}}}%
    \put(0.64798416,0.57107139){\makebox(0,0)[lt]{\lineheight{1.25}\smash{\begin{tabular}[t]{l}$3$\end{tabular}}}}%
    \put(0,0){\includegraphics[width=\unitlength,page=2]{pachner.pdf}}%
    \put(0.87337507,0.7303546){\makebox(0,0)[lt]{\lineheight{1.25}\smash{\begin{tabular}[t]{l}$4$\end{tabular}}}}%
  \end{picture}%
\endgroup%

\caption{A 2--3 Pachner move.}
\label{fig.pachner}
\end{figure}

Recall that the face equation of a face $[i,j,k]$ of a tetrahedron $[i,j,k,l]$
with $i<j<k$ is given by
\be
E^l_{ijk}  : c_{ij} \theta_{ijl}-c_{ik} \theta_{ikl}+ c_{jk} \theta_{jkl} = 0 \, . 
\ee
For $\Delta=[0,2,3,4]$ and $[1,2,3,4]$ we choose $e_\Delta=[2,4]$. Then the
face-matrix $F_{\calT,c}$ of $\calT$ contains a submatrix
\be
\begin{tabular}{c|c|c|c|c|c|c}
& $\theta_{234}$ & others\\
\hline
$E_{024}^3$ & $c_{24}$   & $R_{024}^3$ \\[3pt]
$E_{234}^0$ &    & $R_{234}^0$ \\[3pt]
$E_{124}^3$ &  $c_{24}$ & $R_{124}^3$ \\[3pt]
$E_{234}^1$ &  & $R_{234}^1$
\end{tabular}
\ee
where $R^{l}_{ijk}$ is the row of $E^{l}_{ijk}$ except the $\theta_{234}$-entry.
Using an elementary row operation, we can modify the face-matrix without
changing its determinant as
\be
\label{eqn.table1}
\begin{tabular}{c|c|c|c|c|c|c}
& $\theta_{234}$ & others\\
\hline
$E_{024}^3$ &    & $R_{024}^3-R_{124}^3$ \\[3pt]
$E_{234}^0$ &    & $R_{234}^0$ \\[3pt]
$E_{124}^3$ &  $c_{24}$ & $R_{124}^3$ \\[3pt]
$E_{234}^1$ &  & $R_{234}^1$
\end{tabular}
\ee
For $\Delta=[0,1,2,3]$ (resp., $[0,1,3,4]$ and $[0,1,2,4]$) we choose
$e_\Delta = [2,3]$ (resp., $[3,4]$ and $[2,4]$). Then the face-matrix $F_{\calT',c}$
of $\calT'$ contains a submatrix 
\be
\begin{tabular}{c|c|c|c|c|c|c}
& $\theta_{012}$ & $\theta_{013}$ & $\theta_{014}$ & others\\
\hline
$E_{123}^0$ & $c_{12}$ & $-c_{13}$ &   & $R_{123}^0$ \\[3pt]
$E_{023}^1$ & $c_{02}$ & $-c_{03}$ &   & $R_{023}^1$ \\[3pt]
$E_{134}^0$ &   & $c_{13}$ &  $-c_{14}$ & $R_{134}^0$ \\[3pt]
$E_{034}^1$ &  & $c_{03}$ & $-c_{04}$   & $R_{034}^1$ \\[3pt]
$E_{124}^0$ & $c_{12}$ &  & $-c_{14}$  & $R_{124}^0$ \\[3pt]
$E_{024}^1$ & $c_{02}$ & & $-c_{04}$   & $R_{024}^1$
\end{tabular}
\ee
Preserving the determinant, we apply elementary row operations to obtain:
\be
\label{eqn.table2}
\begin{tabular}{c|c|c|c|c|c|c}
& $\theta_{012}$ & $\theta_{013}$ & $\theta_{014}$ & other 6 faces\\
\hline
$E_{123}^0$ &  &  &   & $R_{123}^0 - \frac{c_{12}}{c_{02}}R_{023}^1
- \frac{c_{23}}{c_{02}} \frac{c_{04}}{c_{34}} R_{134}^0+\frac{c_{23}}{c_{02}}
\frac{c_{14}}{c_{34}}R_{034}^1$ \\[3pt]
$E_{023}^1$ & $c_{02}$ & $-c_{03}$ &   & $R_{023}^1$ \\[3pt]
$E_{134}^0$ &   & $-\frac{c_{34}}{c_{04}} c_{01}$ &   &
$R_{134}^0-\frac{c_{14}}{c_{04}}R_{034}^1$ \\[3pt]
$E_{034}^1$ &  & $c_{03}$ & $-c_{04}$   & $R_{034}^1$ \\[3pt]
$E_{124}^0$ &   &   &   & $R_{124}^0-R_{134}^0-R_{123}^0$ \\[3pt]
$E_{024}^1$ &   &  &   & $R_{024}^1-R_{034}^1-R_{023}^1$
\end{tabular}
\ee
On the other hand, one computes that
\be
\begin{aligned}
R_{234}^0&= R_{124}^0-R_{134}^0-R_{123}^0   \\
R_{234}^1&=R_{024}^1-R_{034}^1-R_{023}^1 \\
R^3_{024}-R^3_{124} &= \frac{c_{02}} {c_{23}}	R_{123}^0
- \frac{c_{12}}{c_{23}}R_{023}^1
-  \frac{c_{04}}{c_{34}} R_{134}^0+ \frac{c_{14}}{c_{34}}R_{034}^1 
\end{aligned}
\ee
It follows that \eqref{eqn.table2} is equal to
\be
\label{eqn.table3}
\begin{tabular}{c|c|c|c|c|c|c}
& $\theta_{012}$ & $\theta_{013}$ & $\theta_{014}$ & others\\
\hline
$E_{123}^0$ &  &  &   & $\frac{c_{23}}{c_{02}} (R^3_{024}-R^3_{124})$\\[3pt]
$E_{023}^1$ & $c_{02}$ & $-c_{03}$ &   & $R_{023}^1$ \\[3pt]
$E_{134}^0$ &   & $-\frac{c_{34}}{c_{04}} c_{01}$ &   &
$R_{134}^0-\frac{c_{14}}{c_{04}}R_{034}^1$ \\[3pt]
$E_{034}^1$ &  & $c_{03}$ & $-c_{04}$   & $R_{034}^1$ \\[3pt]
$E_{124}^0$ &   &   &   & $R_{234}^0$ \\[3pt]
$E_{024}^1$ &   &  &   & $R_{234}^1$
\end{tabular}
\ee
Comparing~\eqref{eqn.table1} and~\eqref{eqn.table3}, we have
\be
 \det F_{\calT,c} = \frac{c_{24}}{(c_{01}c_{23}c_{34})}\det F_{\calT',c}\,.
\ee
One easily checks that the monomial factor in the right-hand side agrees with the
difference coming from the monomial term
$c_1\cdots c_N \prod c(e_\Delta)$ in~\eqref{delta2}. This proves that
$\delta_{\calT,c,2} = \delta_{\calT',c,2}$. As the effect of the insertion of monomials
in $t$ is separated from the above computation, this also proves that
$\delta_{\calT,c,2}(t) = \delta_{\calT',c,2}(t)$.
\end{proof}

\begin{remark}
\label{rem.ptolemy23}
The proof of the above theorem contains the behavior of the Ptolemy variety under
Pachner 2--3 moves. This can be used to show that the determinant of the Jacobian
of the Ptolemy variety, suitable normalized, is invariant under 2--3 Pachner moves,
and conjecturally equal to the 1-loop invariant defined in~\cite{DG1};
see~\cite{Yoon:ptolemy}.
\end{remark}

\begin{theorem}
	\label{thm.lift}
	A Ptolemy assignment $c$ on $\calT$ lifts to a super-Ptolemy assignment
	$(c,\theta)$ with $\th \neq 0$ if and only if
	$\delta_{\calT,c,2}(1)=0$.
\end{theorem}
\begin{proof}
	It is clear that if $\delta_{\calT,c,2}(1) \neq 0$, or equivalently, if $F_c = F_c(1)$ is non-singular,
	then $\theta$ should be zero. Conversely, if $\delta_{\calT,c,2}(1)=0$, then there is a
	nonzero vector $v  \in \BC^{2N}$ with $F_c\, v=0$, and  for any $\eta \neq 0 \in
	\BC$ the pair $(c, \eta v)$ is a super-Ptolemy assignment.
\end{proof}

\begin{conjecture}
\label{conj.dt}
The 1-loop polynomial  $\delta_{\calT,c,2}(t)$ agrees with the $\BC^2$-torsion
polynomial $\tau_{M,\rho,2}(t)$ of $\rho$ up to multiplying signs and monomials in
$t$. Here $\rho: \pi_1(M)\rightarrow \SL_2(\BC)$ is a representation associated
to the Ptolemy assignment $c$.
\end{conjecture}

\subsection{$(m,l)$-deformation}

In this section, we deform the 1-loop invariant as well as the 1-loop polynomial.

We fix a cocycle $\sigma$ that assigns a non-zero complex number to each
short edge of $\Tdot$. Recall Equation~\eqref{eqn.sood} that the face-equations
in~\eqref{eqn.odd} admit a deformation according to $\sigma$.
As in Section~\ref{sub.face}, we choose an edge $e_\Delta$ for each tetrahedron 
$\Delta$ of $\calT$ so that we can choose two face-equations in~\eqref{eqn.sood}.
We then create a $2N \times 2N$ matrix $F^\sigma_c$ so that the chosen 
face-equations for $\theta=(\th_1,\dots,\th_{2N})^t$ take the form
\be
F^\sigma_c \, \theta =0 \, .
\ee

\begin{definition}
For a $\sigma$-deformed Ptolemy assignment $c$ on $\calT$
we define the 1-loop invariant by
\be
\label{eqn.defn2}
\delta_{\calT,c,2}:=\frac{1}{c_1 \cdots c_N}
\left( \prod_{\Delta} \frac{1}{c^\sigma(e_\Delta)} \right) \det F^\sigma_c
\ee
where $c^\sigma(e_\Delta)$ is the value of $c$ on the edge $e_\Delta$
times its $\sigma$-coefficient in~\eqref{eqn.sood}:
\be
\begin{tabular}{c|c|c|c|c|c|c}
$e_{\Delta}$ & [0,1] & [0,2] & [0,3] & [1,2] & [1,3] & [2,3] \\
\hline
$c^\sigma(e_\Delta)$ & $\frac{\sigma^0_{21}}{\sigma^3_{12}} c_{01}$ &
$c_{02}$ &
$\frac{\sigma^3_{01}}{\sigma^2_{01}} c_{03}$ & 
$\frac{\sigma^1_{23}}{\sigma^0_{23}} c_{12}$  &
$c_{13}$ &
$\frac{\sigma^{1}_{03}}{\sigma^2_{03}} c_{23}$ 
\end{tabular}
\ee
\end{definition}

As explained in Section~\ref{sub.poly}, if an infinite cyclic cover of $M$ is given,
we twist the matrix $F_c^\sigma$ to obtain $F_c^\sigma(t)$ by inserting monomials in $t$.

\begin{definition}
For a $\sigma$-deformed Ptolemy assignment $c$ on $\calT$
we define the 1-loop polynomial as
\be
\delta_{\calT,c,2}(t):=\frac{1}{c_1 \cdots c_N}
\left( \prod_{\Delta} \frac{1}{c^\sigma(e_\Delta)} \right) \det F^\sigma_c(t) \, .
\ee
\end{definition}

Repeating the same computation given in Sections~\ref{sub.C2torsion}
and~\ref{sub.poly}, we obtain:

\noindent
1. A $\sigma$-deformed Ptolemy assignment $c$ on $\calT$ lifts to a
super-Ptolemy assignment $(c,\theta)$ with $\th \neq 0$
if and only if $\delta_{\calT,c,2}=0$.

\noindent
2. The 1-loop polynomial $\delta_{\calT,c, 2}(t)$  does not depend on the choice
of edge $e_\Delta$ and is invariant under 2--3 Pachner moves up to scalar
multiplication by non-zero complex numbers.

\begin{conjecture}
\label{conj.1}
The 1-loop polynomial  $\delta_{\calT,c,2}(t)$ agrees with the
$\BC^2$-torsion polynomial $\tau_{M,\rho,2}(t)$ of $\rho$ up to multiplying non-zero
complex numbers and monomials in $t$. Here  $\rho : \pi_1(M)\rightarrow \SL_2(\BC)$
is a representation associated to the $\sigma$-deformed Ptolemy assignment $c$.
\end{conjecture}

When $M$ is hyperbolic and the representation $\rho$ in the above conjecture
is a preferred lift of the geometric representation, then we simply denote the 1-loop
polynomial $\delta_{\calT,c,2}(t)$ by $\delta_{\calT,2}(t)$ and the $\BC^2$-torsion
polynomial $\tau_{M,\rho,2}(t)$ by $\tau_{M,2}(t)$. In this case,
Conjecture~\ref{conj.1} reads
\be
\delta_{\calT,2}(t)\overset{?}{=} \tau_{M,2}(t)
\ee
up to multiplying signs and monomials in $t$  (see Remark~\ref{rmk.reduce}
 below).
 
\begin{remark}
\label{rmk.reduce}
The scalar-multiplication ambiguity in Conjecture~\ref{conj.1} is given by some products
of complex numbers that $\sigma$ assigns to short edges. In particular, if $\sigma$
takes values in $\{\pm1\}$ (for instance, if $M$ is hyperbolic and $\rho$ is a lift
of the geometric representation), then we can replace this scalar-multiplication
ambiguity by sign-ambiguity as in Conjecture~\ref{conj.dt}.
\end{remark}

\subsection{Example: the $4_1$ knot continued}
\label{sub.41b}

In this section we verify Conjecture~\ref{conj.1} for the $4_1$ knot. With the
notation of Section~\ref{sub.41}, we have $e_{\Delta_1} =[0,3]$ and
$e_{\Delta_2}=[1,2]$, hence
\be
\begin{aligned}
& c^\sigma(e_{\Delta_1}) = \frac{\sigma(s_{10})^{-1}}{\sigma(s_{2})} c_{2}
= \frac{c_2}{m}, \qquad
& c^\sigma(e_{\Delta_2})= \frac{\sigma(s_{5})^{-1}}{\sigma(s_{6})} c_{2} = c_2
\end{aligned}
\ee
As explained in \cite[Sec.3.1]{GY21}, we insert monomials in $t$ to the matrix
in Equation~\eqref{eqn.exF} to obtain
\be
F^\sigma_c(t) =
\begin{pmatrix}
m^{-2} c_2 & c_1 & 0 & -m^{-1} c_2\\
c_1 & \ell^{-1} m^{-2} c_2 & -m^{-1} c_2 & 0 \\
-c_2 & c_1 & c_2 t& 0 \\
\ell c_1 & -c_2 & 0 & c_2  t
\end{pmatrix} \,.
\ee
Then we obtain
\be
\delta_{\calT,c,2}(t) =
\frac{1}{c_1 c_2} \frac{m}{c_2^2} \det F_c^\sigma(t) = 
m^{-1}  (t^2-2(m+m^{-1})t + 1)
\ee
which agrees with the $\BC^2$-torsion polynomial of $4_1$ up to $m^{-1}$.

\noindent
In particular, for $m=1$ (and $l=-1$) we obtain the $\BC^2$-torsion polynomial
\be
t^2-4t+1
\ee
for an $\SL_2(\BC)$-lift of the geometric representation of the
$4_1$ knot, in agreement with \texttt{SnapPy}  
\begin{lstlisting}
snappy.Manifold('4_1').hyperbolic_SLN_torsion(2)                             
a^2 - 4.0000000000000000000000000000*a + 0.99999999999999999999999999999
\end{lstlisting}


\section{Further discussion}
\label{sec.deltaodd}

In this paper and in our prior work~\cite{GY21}, we introduced 1-loop polynomials
$\delta_{\calT,2}(t)$ and $\delta_{\calT,3}(t)$ determined, respectively, by the
twisted face-matrix and twisted NZ-matrix of an ideal triangulation
$\calT$ of a 3-manifold $M$, and conjectured to be equal to the $\tau_{M,2}(t)$
and $\tau_{M,3}(t)$ torsion polynomials. 

In this section we explain briefly how to derive 1-loop polynomials that
conjecturally equal to the torsion polynomials $\tau_{M,n}(t)$ for all $n \geq 2$.
Recall that the latter are a sequence of polynomials for $n \geq 2$ associated to
a cusped hyperbolic 3-manifold, and determined by the homology of the infinite
cyclic cover of $M$ twisted by the $(n-1)$-rst symmetric power of an
$\SL_2(\BC)$-lift of the geometric representation. 

The torsion polynomials are closely related to the adjoint reprensentation of
$\PGL_n(\BC)$ which decomposes 
\be
\label{AdPGL}
\mathrm{Ad}(\PGL_n(\BC)) = \oplus_{i=1}^{n-1} \BC^{2 i +1}
\ee
into \emph{odd} dimensional representations of $\SL_2(\BC)$. This decomposition
is not special to $\PGL_n(\BC)$, indeed every complex semisimple Lie group $G$
has a canonical principal $\SL_2(\BC)$ subgroup, and decomposing the adjoint
representation of $G$ as an $\SL_2(\BC)$-representation
\be
\label{AdG}
\mathrm{Ad}(G) = \oplus_{i=1}^r \BC^{2 e_i +1}
\ee
one obtains only \emph{odd} dimensional irreducible representations of $\SL_2(\BC)$,
where $e_i$ are the exponents of $G$~\cite{Kostant}. Using the above
decomposition~\eqref{AdPGL}, one can define for each $n \geq 2$ the product
\be
\label{prodn}
\prod_{i=1}^{n-1} \tau_{M,{2i+1}}(t)
\ee
from which one can extract the odd torsion polynomials $\tau_{M,\text{odd}}(t)$,
see e.g.~\cite[Sec.5]{Porti2015}. 

Using the fact that $\PGL_n(\BC)$-representations can be described by gluing
equations associated to $\PGL_n(\BC)$-type Neumann--Zagier
matrices~\cite{Garoufalidis:gluing}, if one twists these matrices by considering
their lifts to an infinite cyclic cover as was done in~\cite{GY21}, one can define
a 1-loop polynomial $\delta_{\calT,\PGL_{n}}(t)$ which would factor as
$\prod_{i=1}^{n-1} \delta_{M,{2i+1}}(t)$ and would conjecturally equal
to the polynomial~\eqref{prodn}. Doing so, one can obtain the odd 1-loop polynomials
that conjecturally equal to the corresponding odd torsion polynomials.

Likewise, an extension of the decomposition~\eqref{AdG} to high-dimensional
orthosymplectic groups, together with a construction of Neumann--Zagier matrices
that describe representations of 3-manifold groups to orthosymplectic groups, along
with their twisted version would determine even 1-loop polynomials that conjecturally
equal to the corresponding even torsion polynomials.

\subsection*{Acknowledgments}

S.G. wishes to thank Nathan Dunfield for enlightening conversations. 
S.Y. wishes to thank  Teruaki Kitano and Joan Porti for helpful conversations.


\bibliographystyle{hamsalpha}
\bibliography{biblio}
\end{document}